\documentclass[a4paper, 12pt]{article}

\usepackage{amsmath,amssymb,amsthm}
\allowdisplaybreaks

\usepackage{a4wide,mathrsfs,bbold,graphicx}

\usepackage{natbib}
\bibpunct{(}{)}{,}{a}{}{,}
\makeatletter
\renewcommand{\NAT@spacechar}{~}
\makeatother

\usepackage{authblk}

\newcommand{\bil}{\mathopen}
\newcommand{\bir}{\mathclose}
\newcommand{\Ex}{\mathbf{E}}
\newcommand{\Pb}{\mathbf{P}}
\newcommand{\dd}{\,\mathrm{d}}
\newcommand{\ind}{\mathbb{1}}
\newcommand{\N}{\mathcal{N}}
\newcommand{\NN}{\mathbb{N}}
\newcommand{\RR}{\mathbb{R}}
\newcommand{\UU}{\mathbb{U}}
\newcommand{\C}{\mathscr{C}}
\newcommand{\D}{\mathscr{D}}
\newcommand{\argsup}{\mathop{\rm argsup}\limits}
\newcommand{\myliminf}{\mathop{\underline\lim}\limits}

\newtheorem{lemma}{Lemma}
\newtheorem{theorem}{Theorem}
\newtheorem{definition}{Definition}

\begin{document}

\title{On Smooth Change-Point Location Estimation for Poisson Processes}

\author[1]{A.~Amiri}
\author[2]{S.~Dachian}

\affil[1,2]{\small University of Lille, CNRS, UMR 8524 --- Laboratoire Paul Painlev\'e, F--59000 Lille, France}
\affil[2]{\small Tomsk State University, International Laboratory of Statistics of Stochastic Processes and Quantitative Finance, 634050 Tomsk, Russia}

\date{}

\maketitle
\begin{abstract}
We are interested in estimating the location of what we call ``smooth change-point'' from $n$ independent observations of an inhomogeneous Poisson process. The smooth change-point is a transition of the intensity function of the process from one level to another which happens smoothly, but over such a small interval, that its length $\delta_n$ is considered to be decreasing to $0$ as $n\to+\infty$. We show that if $\delta_n$ goes to zero slower than $1/n$, our model is locally asymptotically normal (with a rather unusual rate $\sqrt{\delta_n/n}$), and the maximum likelihood and Bayesian estimators are consistent, asymptotically normal and asymptotically efficient. If, on the contrary,~$\delta_n$ goes to zero faster than $1/n$, our model is non-regular and behaves like a change-point model. More precisely, in this case we show that the Bayesian estimators are consistent, converge at rate $1/n$, have non-Gaussian limit distributions and are asymptotically efficient. All these results are obtained using the likelihood ratio analysis method of Ibragimov and Khasminskii, which equally yields the convergence of polynomial moments of the considered estimators. However, in order to study the maximum likelihood estimator in the case where $\delta_n$ goes to zero faster than $1/n$, this method cannot be applied using the usual topologies of convergence in functional spaces. So, this study should go through the use of an alternative topology and will be considered in a future work.


\end{abstract}
\section{Introduction}

This paper lies within the realm of statistical inference for inhomogeneous Poisson processes. Recall that $X=\bigl(X(t),\ 0\leq t\leq T\bigr)$ is an inhomogeneous Poisson process (on an interval $[0,T]$) of intensity function $\lambda(t)$, $0\leq t\leq T$, if $X(0)=0$ and the increments of~$X$ on disjoint intervals are independent Poisson random variables:
\[
\Pb\bigl\{ X(t)-X(s)=k \bigr\}
=
\frac{\bigl(\int_s^t \lambda(t) \dd t\bigr)^k}{k!}\,
\exp\biggl\{-\int_s^t \lambda(t) \dd t\biggr\}.
\]
The model of inhomogeneous Poisson process is at the same time simple enough to allow the use of the likelihood ratio analysis, and sufficiently reach to modelize various random phenomena in diverse applied fields, such as biology, communication, seismology, astronomy, reliability theory, and so on (see, for example, \citet{CoLew66,Thompson88,SnyMil12,Streit10,Sarkar16}, as well as \citet{ChaFin18}).

We are interested in the problem of estimation of the location $\theta$, where the (elsewhere smooth) intensity function of an inhomogeneous Poisson process switches from one level (say $\lambda_0$) to another (say $\lambda_0+r$). This transition can happen in several ways. The intensity function can switch from $\lambda_0$ to $\lambda_0+r$ instantaneously (change-point case), as for example in
\[
\lambda_\theta(t) = \lambda_0 + r \, \ind_{\{t \geq \theta\}}.
\]
It can also go from $\lambda_0$ to $\lambda_0+r$ smoothly over a small interval of some fixed length $\delta>0$ (smooth case), as for example in
\[
\lambda_\theta(t) = \lambda_0 + \frac{r}{\delta} \, (t-\theta) \, \ind_{\{\theta \leq t < \theta+\delta\}}(t) + r \, \ind_{\{t \geq \theta+\delta\}}(t).
\]
As an intermediate case, we can mention the case of a cusp type singularity, where the intensity function goes from~$\lambda_0$ to $\lambda_0+r$ continuously over a small interval of some fixed length~$\delta$, but has an infinite derivative at some point of this interval, as for example in
\[
\lambda_\theta(t) = \lambda_0 + \frac{r}{\delta^\varkappa} \, (t-\theta)^\varkappa \, \ind_{\{\theta \leq t < \theta+\delta\}}(t) + r \, \ind_{\{t \geq \theta+\delta\}}(t),
\]
where $\varkappa \in (0,1/2)$ is the order of the cusp. The three above intensity functions are illustrated in Figure~\ref{fig-3lambda}.
\begin{figure}[!ht]
\centering
\includegraphics[scale=0.35]{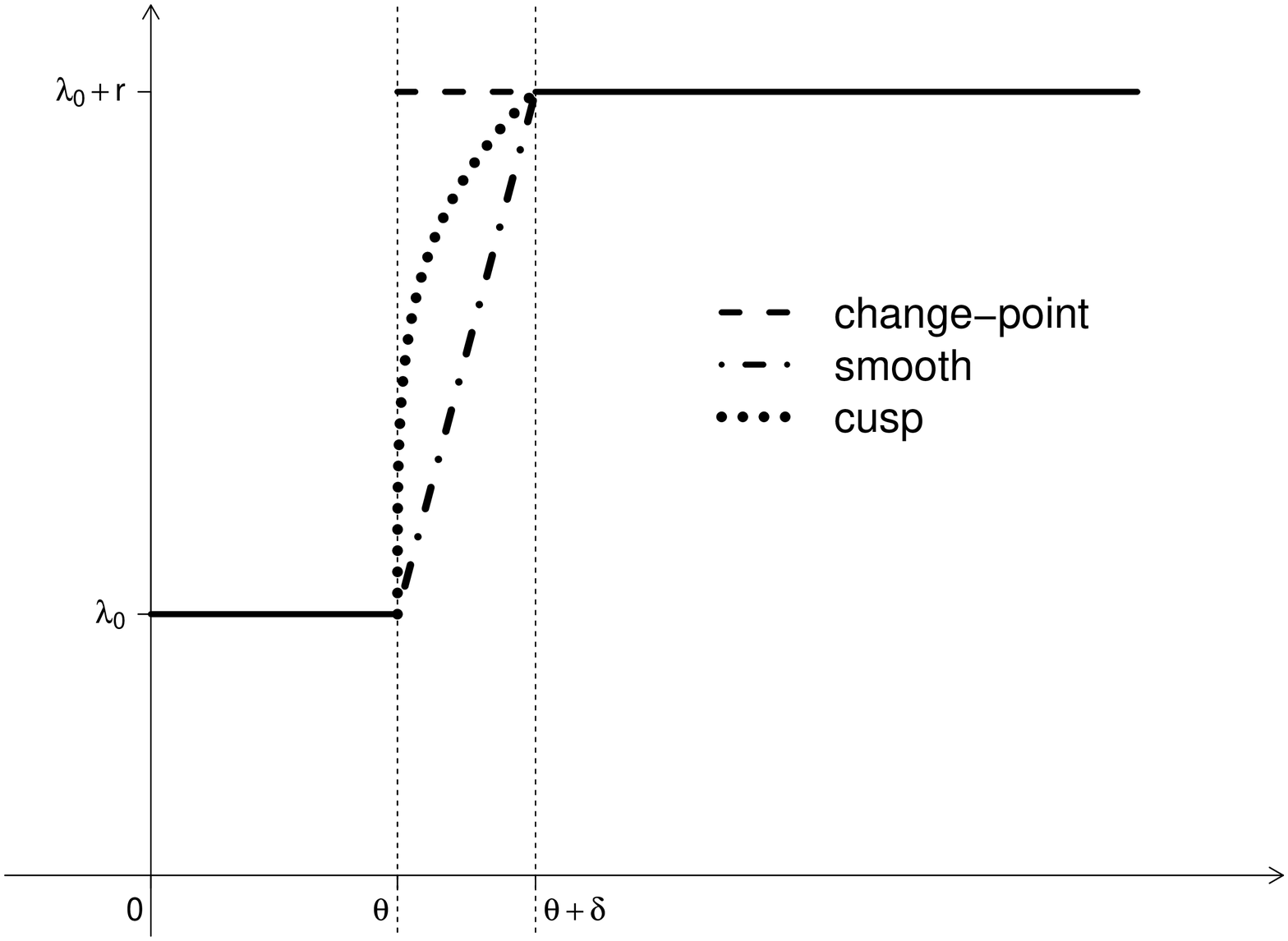}
\caption{change-point, smooth and cusp (with $\varkappa=1/4$) cases}
\label{fig-3lambda}
\end{figure}

In all these cases, the estimation problem is considered in some asymptotic setting, such as $n\to+\infty$ independent observations on a fixed interval, large observation interval asymptotics (an observation on the interval $[0,n\tau]$ with a $\tau$-periodic intensity function), large intensity asymptotics (an observation on a fixed interval with intensity function multiplied by~$n$), and so on.

In the smooth case, the statistical model is regular. The regular statistical models for Poissonian observations were studied by \citeauthor{Kut79} in~\citeyearpar{Kut79,Kut84,Kut98}. It was shown that such models are locally asymptotically normal, and that the maximum likelihood and Bayesian estimators (for any continuous strictly positive prior density~$q$) are consistent, asymptotically normal (with classic rate~$1/\sqrt{n}$) and asymptotically efficient.

In the change-point case, studied by \citeauthor{Kut84} in~\citeyearpar{Kut84,Kut98}, the properties of the estimators are essentially different. The maximum likelihood and Bayesian estimators are consistent, converge at a faster rate $1/n$, their limit distributions are given by some (different) functionals of a two-sided Poisson process, and only the Bayesian estimators are asymptotically efficient.

Finally, the cusp case was studied by \citeauthor{D03S} in~\citeyearpar{D03S}. In this case, the maximum likelihood and Bayesian estimators are consistent, converge at rate $1/n^{(2\varkappa+1)}$ (which is faster than $1/\sqrt{n}$ and slower than $1/n$), their limit distributions are given by some (different) functionals of a two-sided fractional Brownian motion, and only the Bayesian estimators are asymptotically efficient.

Let us note here that all the above cited studies were carried out using the likelihood ratio analysis method introduced by \citeauthor{IbHas81} in~\citeyearpar{IbHas81}, which equally yields the convergence of polynomial moments of the considered estimators.

Note also that recently the problem of source localization on the plane by observations of Poissonian signals from several detectors was considered in all the three cases (smooth, change-point and cusp) in \citet{ChKut20}, \citet*{FKT20} and \citet*{CDK20}, respectively.

In this paper, we consider the situation, which we call \emph{smooth change-point}, where the intensity function goes from $\lambda_0$ to~$\lambda_0+r$ smoothly, but over such a small interval, that its length is considered to be decreasing to $0$ as $n\to+\infty$. Such an intensity function can, for example, be given by
\[
\lambda_{\theta}^{(n)}(t) = \lambda_0 + \frac{r}{\delta_n} \, (t-\theta) \, \ind_{\{\theta \leq t < \theta+\delta_n\}}(t) + r \, \ind_{\{t \geq \theta+\delta_n\}}(t),
\]
where $\delta_n \searrow 0$. Note that the intensity function now depends on~$n$, and so, we are in a scheme of series (triangular array) framework.

The main result of the paper is that there is a ``phase transition'' in the asymptotic behavior of the estimators depending on the rate at which $\delta_n$ goes to $0$. More precisely, we show that if $\delta_n$ goes to zero slower than the ``critical'' rate $1/n$ (that is, if $n\delta_n \to +\infty$), the behavior resembles that of the smooth case, and if $\delta_n$ goes to zero faster than $1/n$ (that is, if $n\delta_n \to 0$), the behavior is exactly the same as in the change-point case. We call these two situations \emph{slow case\/} and \emph{fast case}, respectively.

More specifically, in the slow case we show that our model is locally asymptotically normal, and that the maximum likelihood and Bayesian estimators are consistent, asymptotically normal and asymptotically efficient. It should be noted here that all these asymptotic results use a rather unusual rate $\sqrt{\delta_n/n}$, which is faster than the rate $1/\sqrt{n}$ of the smooth case and slower than the rate $1/n$ of the change-point case. As to the fast case, we show that the asymptotic behavior of the Bayesian estimators is exactly the same as in the change-point model: they are consistent, converge at rate~$1/n$, their limit distribution is given by a functional of a two-sided Poisson process, and they are asymptotically efficient.

In our opinion, these results justify the (successful) use of change-point models for real applications, despite the fact that physical systems can not switch immediately (discontinuously) from one level to another. Indeed, if the transition happens quickly enough, it seems more appropriate to use the fast case of our smooth change-point model, and yet it yields the same asymptotic behavior (at least for the Bayesian estimators, although we conjecture that it is also true for the maximum likelihood estimator). From a more practical point of view, one can say that for a given (large) number~$n$ of observations and a given signal shape (switching from one level to another on an interval of a given length~$\delta$), the (Gaussian) approximation of the estimation errors provided by the regular model is suitable when $n\delta$ is large, and the approximation provided by the change-point model when~$n\delta$ is small.

Let us note that all our results were obtained using the likelihood ratio analysis method of \citeauthor{IbHas81}, which equally yields the convergence of polynomial moments of the considered estimators. On the other hand, for the study of the maximum likelihood estimator, this method needs the convergence of the normalized likelihood ratio in some functional space, and up to the best of our knowledge, until now it was only applied using either the space $\C_0(\RR)$ of continuous functions on~$\RR$ vanishing at~$\pm\infty$ equipped with the topology induced by the usual $\sup$ norm, or the Skorokhod space $\D_0(\RR)$ of c\`adl\`ag functions vanishing at~$\pm\infty$ equipped with the usual Skorokhod topology. However, we will see that in the fast case this convergence can not take place in neither of these topologies, as both of them do not allow the convergence of continuous functions to a discontinuous limit. So, the study of the maximum likelihood estimator in the fast case should go through the use of an alternative topology and will be considered in a future work.

Another possible perspective is to study the behavior of the estimators in the critical case~$\delta_n = c/n$. Also, for models where the transition of the intensity function from $\lambda_0$ to~$\lambda_0+r$ is continuous over an interval of length~$\delta$ but has a cusp of order $\varkappa$ at some point of this interval, it can be interesting to study the situations (somewhat similar to our model) where $\delta = \delta_n \searrow 0$ and/or $\varkappa = \varkappa_n \searrow 0$.

Finally, it can be interesting to consider smooth change-point situations for other models of observations, where we believe it should be possible to obtain similar results. It is worth mentioning that a similar result for the signal in white Gaussian noise~(WGN) model would be consistent with the heuristic considerations of Section~5.3.6 of \citet{TriShi86}, where the authors conclude that for a given smooth (but close to discontinuous) signal, the regular model is more ``adequate'' when the signal-to-noise ratio~(SNR) is large, and the change-point one when the SNR is moderate. Indeed, considering a given signal corresponds to fixing the length of the transition interval $\delta$, and as bigger values of the SNR in the signal in WGN model correspond to bigger values of $n$ in our model, the large SNR case is consistent with the case when $n\delta$ is large, and the moderate SNR case with the case when $n\delta$ is small (but $n$ large).

\section{Statement of the problem}

We consider the model of observation of $n\in\NN^*$ independent realizations of an inhomogeneous Poisson process. Let $0<\alpha < \beta <\tau$ be some known constants, and $\psi$ be some known strictly positive continuous function on $[0,\tau]$. Let also $r>-\min_{0\leq t\leq \tau} \psi (t)$ be some known constant, $({\delta_n})_{n\in\NN}$ be some known sequence decreasing to $0$, and $\theta \in \Theta =(\alpha,\beta)$ be a one-dimensional unknown parameter that we want to estimate in the asymptotics~$n\to+\infty$. We observe $X^{(n)}=(X_1,\ldots, X_n)$, where $X_j=\bigl(X_j(t),\ 0 \leq t \leq \tau \bigr)$, $j = 1,\ldots , n$, are independent Poisson processes on the interval $[0,\tau]$ with intensity function $\lambda_{\theta}=\lambda_{\theta}^{(n)}$, $\theta \in \Theta$, given by
\begin{equation}
\label{lambda}
\lambda_{\theta}^{(n)}(t)
=
\psi(t) + \frac{r}{\delta_n} \, (t-\theta) \, \ind_{\bil[\theta, \theta+\delta_n\bir[}(t) + r \, \ind_{\bil[\theta+\delta_n, \tau\bir]}(t),\qquad 0\leq t\leq \tau.
\end{equation}
This function, in an important particular case $\psi \equiv\lambda_0>0$, is
presented in Figure~\ref{fig-lambda}.
\begin{figure}[!ht]
\centering
\includegraphics[scale=0.35]{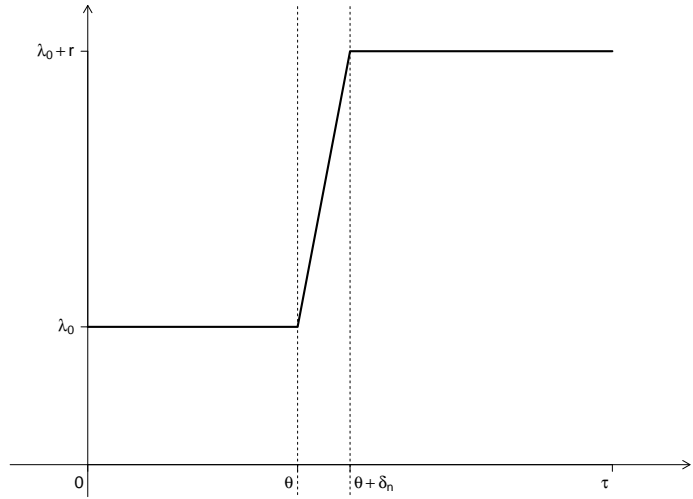}
\caption{intensity function $\lambda_\theta^{(n)}$ with $\psi\equiv\lambda_0$}
\label{fig-lambda}
\end{figure}

Note that our model of observation is equivalent to observing a single realization on~$[0, n\tau]$ of an inhomogeneous Poisson process $X=\bigl(X(t),\ 0\leq t\leq n\tau\bigr)$ with a $\tau$-periodic intensity function equal to $\lambda_\theta^{(n)}$ on the first period (large observation interval asymptotics). Also, it is equivalent to observing a single realization on~$[0, \tau]$ of an inhomogeneous Poisson process $Y^{(n)}=\bigl(Y^{(n)}(t),\ 0\leq t\leq \tau \bigr)$ of intensity function $\Lambda_{\theta}^{(n)} = n \, \lambda_{\theta}^{(n)}$ (large intensity asymptotics).

Recall that regular models of Poissonian observations were treated previously and shown to be locally asymptotically normal~(LAN) by \citeauthor{Kut79} in~\citeyearpar{Kut79,Kut84,Kut98} (see also \citet*{DKY16R}, where the corresponding hypothesis testing problem was considered). An example of such a regular model is the model with an intensity function~$\lambda_\theta$ given as $\lambda_\theta^{(n)}$ in~\eqref{lambda}, but with~$\delta_n$ replaced by a strictly positive constant $\delta$:
\begin{equation}
\label{lambda-reg}
\lambda_{\theta}(t)=
\psi(t)+ \frac{r}{\delta} \, (t-\theta) \, \ind_{\bil[\theta, \theta+\delta\bir[}(t) + r \, \ind_{\bil[\theta+\delta, \tau\bir]}(t),\qquad 0\leq t\leq \tau.
\end{equation}
In this case, the intensity function is continuous and do not depend on $n$.

Recall also that various singular models of Poissonnian observations were already treat\-ed previously. The change-point case was studied by \citeauthor{Kut84} in~\citeyearpar{Kut84,Kut98}, the cusp case was considered by \citeauthor{D03S} in~\citeyearpar{D03S}, and the cases of $0$-type and $\infty$-type singularities were investigated by \citeauthor{D11S} in~\citeyearpar{D11S} (see also \citet*{DKY16S}, where the hypothesis testing problem was considered for different singular cases). An example of a change-point model is the model with an intensity function $\lambda_\theta$ given as $\lambda_\theta^{(n)}$ in~\eqref{lambda}, but with~$\delta_n$ replaced by~$0$:
\begin{equation}
\label{lambda-sing}
\lambda_\theta(t) = \psi(t)+r\, \ind_{\{t\geq \theta\}} ,\qquad 0\leq t\leq \tau.
\end{equation}
In this case, the intensity function is discontinuous and do not depend on $n$.

For our model, the intensity function $\lambda_\theta^{(n)}$ is continuous for all $n \in \NN^*$, but its limit as~$n\to +\infty$ is discontinuous. It is, in some sense, the inverse of the case treated in \citet{DY15}, where the intensity function was supposed to have a discontinuity that disappears as~$n\to +\infty$.

We denote $\Pb_\theta = \Pb_\theta^{(n)}$ the probability measure corresponding to $X^{(n)}$. We also denote $\Ex_\theta = \Ex_\theta^{(n)}$ the corresponding mathematical expectation. The likelihood, with respect to the measure $\Pb_{*} = \Pb_{*}^{(n)}$ corresponding to $n$ independent homogeneous Poisson processes of unit intensity, is given (see, for example, \citet{LiShi01}) by
\begin{align*}
L \bigr(\theta, X^{(n)}\bigl)
&=
\frac{\dd\Pb_{\theta}\bigl(X^{(n)}\bigr)}{\dd\Pb_{*}} \\*
&=
\exp\Biggl\{ \sum_{j=1}^{n} \int_{0}^{\tau} \ln\bigl(\lambda_{\theta}(t)\bigr)\dd X_{j}(t)-n\int_{0}^{\tau} \bigl(\lambda_{\theta}(t)-1 \bigr) \dd t\Biggr\}, \qquad \theta \in \Theta .
\end{align*}

As estimators of the unknown parameter $\theta$, we consider the maximum likelihood estimator~(MLE) and the Bayesian estimators~(BEs). The MLE~$\hat{\theta}_n$ is given by
\[
\hat{\theta}_n
=
\argsup_{\theta \in \Theta} L \bigl(\theta, X^{(n)} \bigr),
\]
and the BE $\tilde{\theta}_n$ for quadratic loss and prior density $q$ is given by
\[
\tilde{\theta}_n
=
\frac{\int_{\alpha}^{\beta}\theta \, q(\theta) \, L \bigl(\theta, X^{(n)} \bigr) \dd \theta}{\int_{\alpha}^{\beta} q(\theta) \, L \bigl(\theta, X^{(n)} \bigr) \dd \theta} \, .
\]

Both in the regular and singular cases cited above, the study of the asymptotic behavior of the MLE and of the BEs was carried out using the likelihood ratio analysis method introduced by \citeauthor{IbHas81} in~\citeyearpar{IbHas81}. This method consist in first studying the \emph{normalized likelihood ratio\/} given by
\begin{align*}
Z_n(u) &= Z_n^{(\theta)}(u) =
\frac{\dd\Pb_{\theta+u\varphi_n}\bigl(X^{(n)}\bigr)}{\dd\Pb_{\theta}}
=
\frac{L \bigl(\theta+u\varphi_n, X^{(n)}\bigr)}{L\bigl(\theta, X^{(n)}\bigr)}\\*
&=
\exp\Biggl\{ \sum_{j=1}^{n} \int_{0}^{\tau} \ln\biggl(\frac{\lambda_{\theta+u\varphi_n}(t)}{\lambda_{\theta}(t)} \biggr)\dd X_{j}(t) - n\int_{0}^{\tau} \bigl(\lambda_{\theta+u\varphi_n}(t) - \lambda_{\theta}(t) \bigr) \dd t\Biggr\},\qquad u\in\UU_n,
\end{align*}
where $\UU_n=\bil]\varphi_n^{-1}(\alpha-\theta),\varphi_n^{-1}(\beta-\theta)\bir[$ and $\varphi_n$ is some sequence decreasing to $0$, called \emph{likelihood normalization rate}. This rate must be chosen so that the process $Z_n$ converges (in some sense) to a non-degenerate (not identically equal to $1$) limit process defined on the whole real line (note that $\UU_n\uparrow\RR$), called \emph{limit likelihood ratio}. Then, the properties of the MLE and of the BEs are deduced.

In the regular case (see \citet{Kut79,Kut84,Kut98}), the likelihood normalization rate can be chosen as
\[
\varphi_n = \frac{1}{\sqrt{n}} \, .
\]
Note that in this case the processes $Z_n$, $n \in \NN$, can be extended to the whole real line so that their trajectories almost surely belong to the space $\C_0(\RR)$ of continuous functions on~$\RR$ vanishing at~$\pm\infty$. The process $Z_n$ converge, in $\C_0(\RR)$ equipped with the usual $\sup$ norm, to the process $Z^\circ_{I(\theta)}$, where
\[
I(\theta) = \int_{0}^{\tau} \frac{\bigl(\dot\lambda_\theta (t)\bigr)^2}{\lambda_\theta (t)} \, \dd t
\]
(here $\dot\lambda_\theta(t)$ denotes the derivative of $\lambda_\theta(t)$ w.r.t.~$\theta$) is the Fisher information, and for any $F \in \RR$, the process $Z^\circ_F$ is defined by
\begin{equation}
\label{LAN_LLR}
Z^\circ_F (u) = \exp\biggl\{ u\,\xi_F - \frac{u^2}{2}\, F \biggr\},\qquad u\in\RR.
\end{equation}
Here and in the sequel $\xi_F\sim \N(0,F)$. In fact, the model is LAN (with classic rate $1/\sqrt{n}$) in this case. Of course, we could have also chosen the rate $\varphi_n=1/\sqrt{nI(\theta)}$, in which case the limit likelihood ratio process would be
\[
Z_1^\circ (u) = \exp\biggl\{ u\,\xi_1 - \frac{u^2}{2} \biggr\},\qquad u\in\RR.
\]

Then, the MLE and the BEs (for any continuous strictly positive prior density $q$) are consistent, are asymptotically normal with rate~$1/\sqrt{n}$:
\[
\sqrt{n} \bigl(\hat{\theta}_n-\theta\bigr) \Longrightarrow \xi_{\frac{1}{I(\theta)}} \quad \text{and} \quad \sqrt{n} \bigl(\tilde{\theta}_n-\theta\bigr) \Longrightarrow \xi_{\frac{1}{I(\theta)}},
\]
we have the convergence of polynomial moments, and both the estimators are asymptotically efficient.  Here and in the sequel, the symbol ``$\Longrightarrow$'' denotes the convergence in distribution (under~$\theta$).

Note also that in the case of the intensity function given by~\eqref{lambda-reg}, using the change of variable
\[
x=\frac{r(t-\theta)}{\delta} \, ,
\]
we obtain that the Fisher information is
\[
I(\theta)
= \int_{\theta}^{\theta +\delta} \frac{\bigl( -\frac{r}{\delta} \bigr)^2}{\psi(t)+\frac{r}{\delta}(t-\theta)} \, \dd t
= \frac{r}{\delta}\int_{0}^{r} \frac{1}{\psi\bigl(\theta+\frac{\delta}{r}\,x\bigr)+x} \, \dd x,
\]
which, in the particular case $\psi\equiv\lambda_0$, amounts to
\[
I(\theta) = \frac{r}{\delta}\, \ln\biggl(\frac{\lambda_0+r}{\lambda_0}\biggr).
\]

As for the change-point case (see \citet{Kut84,Kut98}), the likelihood normalization rate can be chosen as
\[
\varphi_n = \frac{1}{n} \, .
\]
Note that in this case the processes $Z_n$, $n \in \NN$, can be extended to the whole real line so that their trajectories almost surely belong to the Skorokhod space $\D_0(\RR)$ of c\`adl\`ag functions vanishing at~$\pm\infty$. The process $Z_n$ converge to the process $Z^\star_{a,b}$ defined by
\begin{equation}
\label{CP_LLR}
Z^\star_{a,b}(u) =
\begin{cases}
\vphantom{\bigg(} \exp \Bigl\{ \ln \bigl( \frac a b\bigr) Y^+(u) +(b-a)u\Bigr\} , & \text{if }u\in\RR_+, \\
\vphantom{\bigg(} \exp \Bigl\{ \ln \bigl( \frac b a \bigr) Y^-(-u) + (b-a)u\Bigr\} , & \text{if }u\in\RR_-, \\
\end{cases}
\end{equation}
where $a,b>0$ are some constants, and $Y^+$ and $Y^-$ are independent Poisson processes on~$\RR_+$ of constant intensities $b$ and~$a$ respectively. Here, the convergence takes place in~$\D_0(\RR)$ equipped with the usual Skorokhod topology induced by the distance
\[
d(f,g)=\inf_\nu \biggl[ \sup_{u\in\RR} \bigl|f(u)-g\bigl(\nu(u)\bigr)\bigl| + \sup_{u\in\RR} \bigl|u-\nu(u)\bigr| \biggr],
\]
where the $\inf$ is taken over all continuous one-to-one mappings $\nu:\RR \longrightarrow \RR$.

Then, the MLE and the BEs (for any continuous strictly positive prior density $q$) are consistent, converge at rate~$1/n$:
\[
n \bigl(\hat{\theta}_n-\theta\bigr) \Longrightarrow \eta_{a,b} \quad \text{and} \quad n \bigl(\tilde{\theta}_n-\theta\bigr) \Longrightarrow \zeta_{a,b},
\]
where
\[
\eta_{a,b}=\argsup_{u\in \RR} Z^\star_{a,b}(u)
\]
and
\begin{equation}
\label{zeta_ab}
\zeta_{a,b}=\frac{\int_{u\in \RR} u Z^\star_{a,b}(u) \dd u}{\int_{u \in \RR} Z^\star_{a,b} (u) \dd u} \, ,
\end{equation}
we have the convergence of polynomial moments, and the BEs are asymptotically efficient.

Note also that in the case of the intensity function given by~\eqref{lambda-sing}, we have $a=\psi(\theta)$ and~$b=\psi(\theta)+r$, which, in the particular case $\psi\equiv\lambda_0$, amounts to $a=\lambda_0$ and $b=\lambda_0+r$.

Let us also mention here that the random variables $\eta_{a,b}$ and $\zeta_{a,b}$ can be rewritten as
\[
\eta_{a,b}=\frac{\eta_\rho}{a-b}\quad\text{and}\quad
\zeta_{a,b}=\frac{\zeta_\rho}{a-b}\,,
\]
where $\rho=\bigl|\ln\bigl(\frac ab\bigr)\bigr|$ and the random variables $\eta_\rho$ and $\zeta_\rho$ are defined in the same way as $\eta_{a,b}$ and $\zeta_{a,b}$, but using the random process $Z^\star_\rho(u)$, $u\in\RR$, given by
\[
Z^\star_\rho(x)=\begin{cases}
\vphantom{\Big)}\exp\bigl\{\rho\,\Pi^+(x)-x\bigr\}, &\text{if } x\geq 0,\\
\vphantom{\Big)}\exp\bigl\{-\rho\,\Pi^-(-x)-x\bigr\}, &\text{if } x\leq 0,\\
\end{cases}
\]
with $\Pi^+$ and $\Pi^-$ independent Poisson processes on $\RR_+$ of constant intensities $1/(e^\rho-1)$ and~$1/(1-e^{-\rho})$ respectively. The approximate values of the second moments of the random variables $\eta_\rho$ and $\zeta_\rho$ (giving the limiting mean squared errors of the estimators) where obtained with the help of numerical simulations by \citeauthor{D10S} in~\citeyearpar{D10S}. Moreover, it was shown that $\Ex \eta_\rho^2 \sim 2$ and $\Ex \zeta_\rho^2 \sim 1$ for large values of $\rho$, as well as that $\Ex \eta_\rho^2 \sim 26/\rho^2$ and~$\Ex \zeta_\rho^2 \sim 16\,\zeta(3)/\rho^2$ (where $\zeta$ is Riemann's zeta function) for small values of $\rho$.

Finally, let us note that for our model, the trajectories of the (extended to the whole real line) processes $Z_n$, $n \in \NN$, almost surely belong to the space $\C_0(\RR)$. We will see later in this paper that in the case $ n \delta_n\to +\infty$, the trajectories of the limit process also belong to $\C_0(\RR)$, and the convergence takes place in this space equipped with the usual $\sup$ norm. However, this is not the case when $n\delta_n\to 0$. In this case, the trajectories of the limit process must be discontinuous $\bigl($belong to $\D_0(\RR)\setminus\C_0(\RR)\bigr)$, and hence the convergence can not take place neither in the topology induced by the $\sup$ norm, nor in the usual Skorokhod topology.

\section{Main results}

It turns out that the asymptotic behavior of our model depends on the rate of convergence of~$\delta_n$ to zero. More precisely, there are three different cases:
\begin{align}
&n\delta_n\xrightarrow[n\to +\infty]{} +\infty , \label{casl} \\*
&n\delta_n\xrightarrow[n\to +\infty]{} 0 \label{casr} \\*
\intertext{and}
&n\delta_n\xrightarrow[n\to +\infty]{} c>0. \label{casc}
\end{align}
In this paper, we limit ourselves to the study of the MLE and of the BEs in the case~\eqref{casl}, which we call \emph{slow case}, and to the study of the BEs in the case~\eqref{casr}, which we call \emph{fast case}.

The study of the MLE in the fast case is more complicated, due to the already mentioned fact that in this case the convergence of the normalized likelihood ratio can not take place neither in the topology induced by the $\sup$ norm, nor in the usual Skorokhod topology. So, the MLE in the fast case, as well as the case~\eqref{casc}, will be considered in future works.

\subsubsection*{Slow case}

In order to study the behavior of the MLE and of the BEs of~$\theta$ in the slow case, we choose the likelihood normalization rate
\[
\varphi_n=\sqrt{\frac{\delta_n}{n}} \, ,
\]
and we denote
\[
F = F(\theta) = r\ln \Bigl( \frac{\psi(\theta)+r}{\psi(\theta)} \Bigr).
\]

We also recall the random process $Z^\circ_F$ defined by~\eqref{LAN_LLR} and note that
\[
\argsup_{u\in \RR} Z^\circ_F(u) = \frac{\int_{u\in \RR} u Z^\circ_{F}(u) \dd u}{\int_{u \in \RR} Z^\circ_{F}(u) \dd u} = \frac{\xi_F}{F} \sim \N\Bigl(0, \frac1F\Bigr).
\]

Note that the rate $\varphi_n$ goes to zero faster than $1/\sqrt{n}$ (the classic rate of the regular case) and, as it follows from~\eqref{casl}, slower than $1/n$ (the rate of the change-point case).

Now we can state the following theorem giving a H\'ajek-Le Cam lower bound on the mean squared errors of all the estimators.
\begin{theorem}
\label{Borne.CL}
Suppose\/ $n\delta_n\to+\infty$. Then, for all\/ $\theta_0 \in \Theta$, we have
\[
\lim_{\epsilon \to 0}\ \myliminf_{n\to +\infty}\ \inf_{\bar{\theta}_n}\ \sup_{\bil| \theta-\theta_0 \bir|<\epsilon} \varphi_n^{-2} \ \Ex_{\theta} \bigl(\bar{\theta}_n- \theta\bigr)^2 \geq \frac{1}{F(\theta_0)} \, ,
\]
where the\/ $\inf$ is taken over all possible estimators\/ $\bar{\theta}_n$ of the parameter\/ $\theta$.
\end{theorem}

This theorem is a direct consequence of the fact that our model is LAN (though with a rather unusual rate $\varphi_n= \sqrt{\delta_n/n}$). More precisely, we have the following lemma, which will be proved in Section~\ref{proofs} by checking the conditions of Theorem~2.1 of \citet{Kut98} and applying it.
\begin{lemma}
Suppose\/ $n\delta_n\to+\infty$. Then, the normalized (using the rate\/ $\varphi_n= \sqrt{\delta_n/n}$) likelihood ratio\/ $Z_n(u) = \frac{\dd\Pb_{\theta+u\varphi_n}(X^{(n)})}{\dd\Pb_{\theta}}\,$, $u\in\UU_n$, admits the representation
\[
Z_n(u) = \exp\biggl\{ u\Delta_n - \frac{u^2}{2} \, F + \varepsilon_n(\theta,u)\biggr\},
\]
where\/ $\Delta_n \Longrightarrow \xi_F$ is given by
\[
\Delta_n(\theta) =
- \frac{r}{\sqrt{n \delta_n}} \; \sum_{j=1}^{n} \int_{\theta}^{\theta+\delta_n} \frac{1}{\psi(\theta)+\frac{r}{\delta_n} \, (t-\theta)} \, \dd X_j(t)
+
r\sqrt{n\delta_n},
\]
and\/ $\varepsilon_n(\theta,u)$ converges to zero in probability.
\label{LAN}
\end{lemma}

As it is usual in LAN situations, Theorem~\ref{Borne.CL} allows us to introduce the following definition of H\'ajek-Le Cam efficiency.
\begin{definition}
Suppose\/ $n\delta_n\to+\infty$. We say that an estimator\/ $\theta_n^*$ of the parameter\/ $\theta$ is asymptotically efficient if, for all\/ $\theta_0 \in \Theta$, we have
\[
\lim_{\epsilon \to 0}\ \lim_{n\to +\infty}\ \sup_{\bil| \theta-\theta_0 \bir|<\epsilon} \varphi_n^{-2} \ \Ex_{\theta} \bigl(\theta_n^* - \theta\bigr)^2 = \frac{1}{F(\theta_0)} \, .
\]
\end{definition}

Finally, the asymptotic properties of the MLE and of the BEs are given by the following theorem.
\begin{theorem}
\label{MLE-BE.CL}
Suppose\/ $n\delta_n\to+\infty$. Then, the MLE\/ $\hat{\theta}_n$ and, for any continuous and strictly positive prior density\/ $q$, the BE\/ $\tilde{\theta}_n$ have the following proprieties:
\begin{itemize}
\item $\hat{\theta}_n$ and\/ $\tilde{\theta}_n$ are consistent,
\item $\hat{\theta}_n$ and\/ $\tilde{\theta}_n$ are asymptotically normal with rate\/ $\varphi_n= \sqrt{\delta_n/n}$ and limit variance\/ $1/F$, that is:
\[
\varphi_n^{-1} \bigl(\hat{\theta}_n-\theta\bigr) \Longrightarrow \xi_{1/F}
\quad\text{and\/}\quad
\varphi_n^{-1} \bigl(\tilde{\theta}_n-\theta\bigr) \Longrightarrow \xi_{1/F} \, ,
\]
\item we have the convergence of polynomial moments, that is, for any\/ $p>0$, we have
\[
\lim_{n\to +\infty} \varphi_n^{-p} \, \Ex_\theta \bigl|\hat{\theta}_n-\theta\bigr|^p = \Ex {\bil| \xi_{1/F} \bir|}^p
\quad\text{and\/}\quad
\lim_{n\to +\infty} \varphi_n^{-p} \, \Ex_\theta \bigl|\tilde{\theta}_n-\theta\bigr|^p = \Ex {\bil| \xi_{1/F} \bir|}^p,
\]
\item $\hat{\theta}_n$ and\/ $\tilde{\theta}_n$ are asymptotically efficient.
\end{itemize}
\end{theorem}

Note that the LAN property yields, in particular, the convergence of finite-dimensional distributions of the normalized likelihood ratio process $Z_n$ to those of the process~$Z^\circ_F$. So, in order to prove Theorem~\ref{MLE-BE.CL}, it is sufficient to establish two additional lemmas (which will be done in Section~\ref{proofs}) and apply Theorems~1.10.1 and~1.10.2 of \citet{IbHas81}.

\subsubsection*{Fast case}

In order to study the behavior of Bayesian estimators of $\theta$ in the fast case, we choose the likelihood normalization rate
\[
\varphi_n=\frac{1}{n} \, ,
\]
and we denote
\[
a=a(\theta)=\psi(\theta) \quad \text{and} \quad b=b(\theta)=\psi(\theta)+r.
\]
We also recall the random process $Z^\star_{a,b}$ defined by~\eqref{CP_LLR} and the random variable $\zeta_{a,b}$ defined by~\eqref{zeta_ab}.

Now we can state the following theorem giving a H\'ajek-Le Cam type lower bound on the mean squared errors of all the estimators.
\begin{theorem}
\label{Borne.CR}
Suppose\/ $n\delta_n\to 0$. Then, for all\/ $\theta_0 \in \Theta$, we have
\[
\lim_{\epsilon \to 0}\ \myliminf_{n\to +\infty}\ \inf_{\bar{\theta}_n}\ \sup_{\bil| \theta-\theta_0 \bir|<\epsilon} n^2 \ \Ex_{\theta} (\bar{\theta}_n- \theta)^2 \geq \Ex \zeta_{a(\theta_0),b(\theta_0)}^2,
\]
where the\/ $\inf$ is taken over all possible estimators\/ $\bar{\theta}_n$ of the parameter\/ $\theta$.
\end{theorem}

This theorem allows us to introduce the following definition.
\begin{definition}
Suppose\/ $n\delta_n\to 0$. We say that an estimator\/ $\theta_n^*$ of the parameter\/ $\theta$ is asymptotically efficient if, for all\/ $\theta_0 \in \Theta$, we have
\[
\lim_{\epsilon \to 0}\ \lim_{n\to +\infty}\ \sup_{\bil| \theta-\theta_0 \bir|<\epsilon} n^2 \ \Ex_{\theta} ( \theta_n^* - \theta)^2 = \Ex \zeta_{a(\theta_0),b(\theta_0)}^2.
\]
\end{definition}

Finally, the asymptotic properties of the BEs are given by the following theorem.
\begin{theorem}
\label{BE.CR}
Suppose\/ $n\delta_n\to 0$. Then, for any continuous and strictly positive prior density\/ $q$, the BE\/ $\tilde{\theta}_n$ has the following proprieties:
\begin{itemize}
\item $\tilde{\theta}_n$ is consistent,
\item $\tilde{\theta}_n$ converges at rate\/ $\varphi_n = 1/n$ and its limit law is that of the random variable\/ $\zeta_{a,b}$, that is:
\[
n \bigl(\tilde{\theta}_n-\theta\bigr) \Longrightarrow \zeta_{a,b},
\]
\item we have the convergence of polynomial moments, that is, for any\/ $p>0$, we have
\[
\lim_{n\to +\infty} n^p \, \Ex_\theta \bigl|\tilde{\theta}_n-\theta\bigr|^p = \Ex {\bil| \zeta_{a,b} \bir|}^p,
\]
\item $\tilde{\theta}_n$ is asymptotically efficient.
\end{itemize}
\end{theorem}

The argument behind Theorems~\ref{Borne.CR} and~\ref{BE.CR} is the likelihood ratio analysis method of \citeauthor{IbHas81}. However, as we already mentioned, the convergence of the normalized likelihood ratio $Z_n$ to the limit likelihood ratio $Z^\star_{a,b}$ can not take place neither in the topology induced by the $\sup$ norm, nor in the usual Skorokhod topology. Nevertheless, the convergence in a functional space is only needed for the properties of the MLE; in order to obtain the lower bound and the properties of the BEs, it is sufficient to establish the convergence of finite-dimensional distributions together with two additional lemmas (which will be done in Section~\ref{proofs}) and apply Theorems~1.9.1 and~1.10.2 of \citet{IbHas81}.

\section{Proofs}
\label{proofs}

In order to simplify the exposition and make the ideas of the proofs clearer, we present them in the particular case $\psi\equiv \lambda_0$. Note that as in singular problems all the information usually comes from the vicinity of the singularity, this is not a real loss of generality. Moreover, the given proofs can be easily extended to the general case (some details are given where necessary).

\subsubsection*{Slow case}

As we already explained above, Lemma~\ref{LAN} will be proved by applying Theorem~2.1 of \citet{Kut98}. So, we need to check the conditions of this theorem.

For this, remind that our model of observation is equivalent to observing a single realization on $[0, \tau]$ of a Poisson process $Y^{(n)}=\bigl(Y^{(n)}(t),\ 0\leq t\leq \tau \bigr)$ of intensity function~$\Lambda_{\theta}^{(n)} = n \, \lambda_{\theta}^{(n)}$. The process $Y^{(n)}$ can be defined, for example, by $Y^{(n)}(t) = \sum_{j=1}^{n} X_j(t)$.

Denote
\[
S_n(\theta_1,\theta_2, t) = \frac{\Lambda_{\theta_1}^{(n)}(t)}{\Lambda_{\theta_2}^{(n)}(t)} = \frac{\lambda_{\theta_1}^{(n)}(t)}{\lambda_{\theta_2}^{(n)}(t)}
\]
for all $n\in \NN$, $t \in [0,\tau]$ and $\theta_1,\theta_2 \in \Theta$.

The conditions of Theorem~2.1 of \citet{Kut98} are now the following.

\paragraph*{\boldmath$(A_1)$}
For all $n\in \NN$, the intensity measures corresponding to (having them as Radon-Nikodym derivatives) the intensity functions $\Lambda_{\theta}^{(n)}$, $\theta \in \Theta$, are equivalent.

\paragraph*{\boldmath$(A_2)$}
For all $n \in \NN$, there exist some function $q_n$ on $\Theta \times [0,\tau]$, such that
\[
Q_n(\theta)
=
\int_0^\tau q_n(\theta,t)^2 \Lambda_\theta^{(n)}(t) \dd t
\]
is positive for all $\theta \in \Theta$ and, for any $\epsilon>0$, it holds
\[
\int_0^\tau \Bigl| Q_n^{-\frac{1}{2}} (\theta) q_n(\theta,t) \Bigr|^2 \ind_{ \bigl\{ \bigl| Q_n^{-\frac{1}{2}} (\theta) q_n(\theta,t) \bigr|>\epsilon\bigr\} } \Lambda_{\theta}^{(n)} (t) \dd t \xrightarrow[n\to +\infty]{} 0.
\]
\paragraph*{\boldmath$(A_3)$}
Let $\theta_u=\theta + u \, Q_n^{-\frac{1}{2}}(\theta)$ and $\UU'_n = \{u: \, \theta_u \in \Theta\}$. For any $u \in \RR$, we have $u \in \UU'_n$ for~$n$ sufficiently large, and it holds
\begin{align}
&\int_0^\tau
\Bigl[ \ln S_n(\theta_u,\theta,t) - u \, Q_n^{-\frac{1}{2}} (\theta) q_n(\theta,t) \Bigr]^2 \Lambda_{\theta}^{(n)} (t) \dd t \xrightarrow[n\to +\infty]{} 0 \label{A31} \\*
\intertext{and}
&\int_0^\tau
\biggl[
S_n(\theta_u,\theta,t) -1- \ln S_n(\theta_u,\theta,t) -\frac{1}{2} \Bigl( u \, Q_n^{-\frac{1}{2}} (\theta) q_n(\theta,t) \Bigr)^2
\biggr]
\Lambda_{\theta}^{(n)}(t) \dd t
\xrightarrow[n\to +\infty]{} 0 . \label{A32}
\end{align}
\paragraph{}
Since the intensity functions $\Lambda_{\theta}^{(n)}$, $\theta \in \Theta$, are strictly positive, the condition $(A_1)$ is trivially verified.

Now, in order to prove the conditions $(A_2)$ and $(A_3)$, we put
\[
q_n(\theta,t)
=
\frac{-\frac{r}{\delta_n}}{\lambda_0 + \frac{r}{\delta_n} \, (t-\theta)} \,\ind_{\bil[\theta, \theta+\delta_n\bir]} (t),
\]
and so
\[
Q_n(\theta)
=
n \int_\theta^{\theta+\delta_n} \frac{\bigl( -\frac{r}{\delta_n} \bigr)^2}{\lambda_0 + \frac{r}{\delta_n} \, (t-\theta)} \, \dd t
=
r \, \frac{n}{\delta_n} \int_{\lambda_0}^{\lambda_0+r} \frac1x \, \dd x
=
r \ln \Bigl( \frac{\lambda_0+r}{\lambda_0} \Bigr)\frac{n}{\delta_n}
=
F\varphi_n^{-2},
\]
where we used the change of variable
\[
x=\lambda_0 + \frac{r}{\delta_n} \, (t-\theta).
\]
Therefore, we have
\[
Q_n^{-\frac{1}{2}} (\theta)
=
\frac{1}{\sqrt{r \ln \bigl( \frac{\lambda_0+r}{\lambda_0} \bigr)}}\sqrt{\frac{\delta_n}{n}}
=
\frac{\varphi_n}{\sqrt{F}}
=
\gamma \varphi_n,
\]
where we denoted
\[
\gamma
=
\frac{1}{\sqrt{F}} \, .
\]

\begin{proof}[Proof of $(A_2)$]
Let $\theta \in \Theta$ and $\epsilon>0$. Denoting
\[
F_n=
\int_0^\tau \Bigl| Q_n^{-\frac{1}{2}} (\theta) q_n(\theta,t) \Bigr|^2 \ind_{ \bigl\{ \bigl| Q_n^{-\frac{1}{2}} (\theta) q_n(\theta,t) \bigr|>\epsilon\bigr\} } \Lambda_{\theta}^{(n)} (t) \dd t,
\]
we need to prove that
\[
F_n \xrightarrow[n\to +\infty]{} 0.
\]

We can write
\begin{align*}
F_n
&=
\frac{(r\gamma)^2}{\delta_n} \int_\theta^{\theta+\delta_n} \frac{1}{\lambda_0 + \frac{r}{\delta_n} \, (t-\theta)} \, \ind_{ \Bigl\{ \bigl|- \, \frac{1}{\sqrt{n\delta_n}} \, \frac{r\gamma}{\lambda_0 + \frac{r}{\delta_n} \, (t-\theta)} \bigr| \, >\epsilon \Bigr\} } \dd t \\*
&=
\frac{1}{\ln \bigl( \frac{\lambda_0 +r}{\lambda_0} \bigr)} \int_{\lambda_0}^{\lambda_0 +r} \frac{1}{x} \, \ind_{\bigl\{ \frac{\bil|r\bir| \gamma}{\epsilon \sqrt{n\delta_n}} > x \bigr\} } \dd x,
\end{align*}
where we used again the change of variable
\[
x=\lambda_0 + \frac{r}{\delta_n} \, (t-\theta).
\]

Since
\[
n\delta_n \xrightarrow[n\to +\infty]{} +\infty,
\]
we have
\[
F_n = 0
\]
for $n$ sufficiently large, and so we get the condition $(A_2)$.
\end{proof}
\begin{proof}[Proof of $(A_3)$]
First of all, let us note that $\UU'_n=\bil]\varphi_n^{-1} \gamma^{-1} (\alpha-\theta),\varphi_n^{-1} \gamma^{-1} (\beta-\theta)\bir[$, and as~$\varphi_n \searrow 0$, we get $\UU'_n \uparrow \RR$. Thus, for any $u\in\RR$, we have $u\in\UU'_n$ for $n$ sufficiently large.

Further we consider the case where $u>0$ and $r>0$ only (the other cases can be treated in a similar way).

Note that since we are in the slow case, we have $\gamma u\varphi_n < \delta_n$ for $n$ sufficiently large. We present in Figure~\ref{fig-int-slow-case} the functions $\lambda_\theta$ and $\lambda_{\theta_u}$, where $\theta_u = \theta + u \, Q_n^{-\frac{1}{2}}(\theta) = \theta + \gamma u \varphi_n$, with~$u,r>0$ and $n\gg 1$.
\begin{figure}[!ht]
\centering
\includegraphics[scale=0.36]{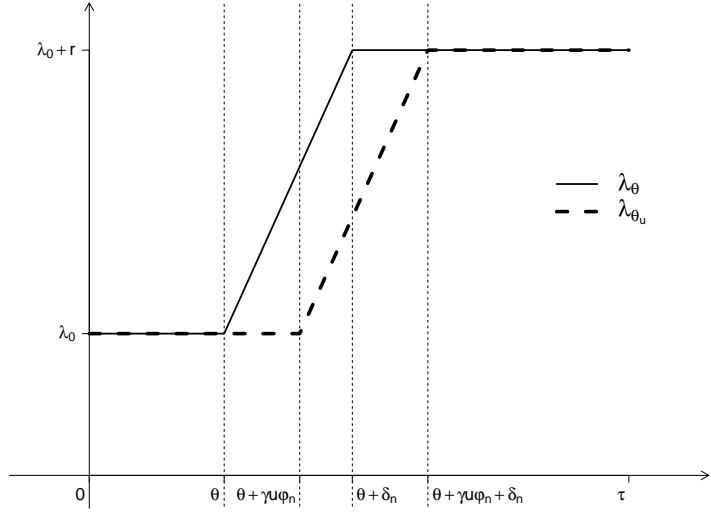}
\caption{$\lambda_\theta$ and $\lambda_{\theta_u}$ with $u,r>0$ and $n\gg 1$ (slow case)}
\label{fig-int-slow-case}
\end{figure}

In order to prove~\eqref{A31}, we denote
\[
G_n
=
\int_0^\tau
\Bigl[ \ln S_n(\theta_u,\theta,t) - u \, Q_n^{-\frac{1}{2}} (\theta) q_n(\theta,t) \Bigr]^2 \Lambda_{\theta}^{(n)} (t) \dd t .
\]
We can write
\begin{align*}
G_n
&=
n\int_\theta^{\theta+\gamma u\varphi_n} \Bigl( \lambda_0+\frac{r}{\delta_n} \, (t-\theta) \Bigr) \biggl[ \ln\biggl( \frac{\lambda_0}{\lambda_0+\frac{r}{\delta_n} \, (t-\theta)}\biggr) + \frac{\varphi_n}{\delta_n} \,\frac{\gamma ru}{\lambda_0+\frac{r}{\delta_n} \, (t-\theta)} \biggr] ^2 \!\dd t \\*
&\phantom{{}=}+
n\int_{\theta+\gamma u\varphi_n}^{\theta+\delta_n}
\Bigl(\lambda_0+\frac{r}{\delta_n} \, (t-\theta) \Bigr)
\biggl[
\ln\biggl( \frac{\lambda_0+\frac{r}{\delta_n} \, (t-\theta-\gamma u\varphi_n)}{\lambda_0+\frac{r}{\delta_n} \, (t-\theta)}\biggr)
+
\frac{\varphi_n}{\delta_n} \,\frac{\gamma ru}{\lambda_0+\frac{r}{\delta_n} \, (t-\theta)}
\biggr]^2
\!\dd t \\*
&\phantom{{}=}+
n\int_{\theta+\delta_n}^{\theta+\delta_n+\gamma u\varphi_n} (\lambda_0+r) \, \ln^2 \Biggl( \frac{\lambda_0+\frac{r}{\delta_n} \, (t-\theta-\gamma u\varphi_n)}{\lambda_0+r}\Biggr) \dd t \\*
&=
B_1+B_2+B_3
\end{align*}
with evident notations, and so it is sufficient to show that $B_1$, $B_2$ and~$B_3$ go to zero.

Let us verify, for example, that
\[
B_1 \xrightarrow[n\to +\infty]{} 0.
\]
Using the change of variable
\[
x=\frac{t-\theta}{\gamma u\varphi_n} \, ,
\]
we obtain
\[
B_1
=
\gamma un\varphi_n \int_0^1 \biggl(\lambda_0 +\gamma ru \, \frac{\varphi_n}{\delta_n} \, x \biggr)
\biggl[
\ln \biggl(1+\frac{\gamma ru}{\lambda_0} \, \frac{\varphi_n}{\delta_n} \, x \biggr)
-
\frac{\varphi_n}{\delta_n} \, \frac{\gamma ru}{\lambda_0+\gamma ru \, \frac{\varphi_n}{\delta_n} \, x}
\biggr]^2
\dd x.
\]
Since, for $0\leq x \leq 1$, we have
\[
\lambda_0 +\gamma ru \, \frac{\varphi_n}{\delta_n} \, x \leq \lambda_0+r
\]
and
\begin{align*}
\biggl[
\ln
\biggl(
1+\frac{\gamma ru}{\lambda_0} \, \frac{\varphi_n}{\delta_n} \, x
\biggr)
-
\frac{\varphi_n}{\delta_n} \, \frac{\gamma ru}{\lambda_0+\gamma ru \, \frac{\varphi_n}{\delta_n} \, x}
\biggr]^2
&\leq
\ln^2 \biggl(1+\frac{\gamma ru}{\lambda_0} \, \frac{\varphi_n}{\delta_n} \, x \biggr)
+
\biggl[
\frac{\varphi_n}{\delta_n} \, \frac{\gamma ru}{\lambda_0+\gamma ru \, \frac{\varphi_n}{\delta_n} \, x}
\biggr]^2 \\*
&\leq
\biggl[
\frac{\gamma ru}{\lambda_0} \, \frac{\varphi_n}{\delta_n} \, x
\biggr]^2
+
\biggl[
\frac{\varphi_n}{\delta_n} \, \frac{\gamma ru}{\lambda_0+\gamma ru \, \frac{\varphi_n}{\delta_n} \, x}
\biggr]^2, \\*
&\leq
2 \, \frac{(\gamma ru)^2}{\lambda_0^2}\, \frac{\varphi_n^2}{\delta_n^2} \, ,
\end{align*}
it comes
\[
B_1 \leq 2 \, \frac{(\lambda_0+r) (\gamma u)^3 r^2}{\lambda_0^2} \, \frac{n \varphi_n^3}{\delta_n^2} \,.
\]
Finally, as
\[
\frac{n \varphi_n^3}{\delta_n^2}
=\frac{1}{\sqrt{n\delta_n}}
\xrightarrow[n\to +\infty]{} 0,
\]
we get
\[
B_1 \xrightarrow[n\to +\infty]{} 0.
\]

Proceeding similarly, it is not difficult to verify that $B_2$ and $B_3$ go to zero as well, and~so~\eqref{A31} is proved.

Now, in order to prove~\eqref{A32}, we denote
\[
H_n = \int_0^\tau h_n(t) \dd t,
\]
where
\[
h_n(t) = \biggl[ S_n(\theta_u,\theta,t) -1- \ln S_n(\theta_u,\theta,t) -\frac{1}{2} \Bigl( u \, Q_n^{-\frac{1}{2}} (\theta) q_n(\theta,t) \Bigr)^2 \biggr] \Lambda_{\theta}^{(n)} (t).
\]
We have
\begin{align*}
\bil|H_n\bir|
&\leq
\int_\theta^{\theta+\gamma u\varphi_n} \bigl|h_n(t)\bigr| \dd t
+
\int_{\theta+\gamma u\varphi_n}^{\theta+\delta_n} \bigl|h_n(t)\bigr| \dd t
+
\int_{\theta+\delta_n}^{\theta+\gamma u\varphi_n+\delta_n} \bigl|h_n(t)\bigr| \dd t \\*
&=
D_1+D_2+D_3
\end{align*}
with evident notations, and so it is sufficient to show that $D_1$, $D_2$ and $D_3$ go to zero.

Let us verify, for example, that
\[
D_1 \xrightarrow[n\to +\infty]{} 0.
\]
Using the change of variable
\[
x=\frac{t-\theta}{\gamma u\varphi_n} \, ,
\]
we obtain
\begin{align*}
D_1
&=
n\int_\theta^{\theta+\gamma u\varphi_n}
\Bigl( \lambda_0+\frac{r}{\delta_n} \, (t-\theta) \Bigr) \,
\Biggl|
\frac{\lambda_0}{\lambda_0+\frac{r}{\delta_n} \, (t-\theta)}
-
1
-
\ln\biggl( \frac{\lambda_0}{\lambda_0+\frac{r}{\delta_n} \, (t-\theta)}\biggr) \\*
&\phantom{{}=n\int_\theta^{\theta+\gamma u\varphi_n} \Bigl( \lambda_0+\frac{r}{\delta_n} \, (t-\theta) \Bigr)
\biggl|}
-\frac{(\gamma ru)^2}{2} \, \frac{\varphi_n^2}{\delta_n^2} \, \frac{1}{\bigr( \lambda_0+\frac{r}{\delta_n} \, (t-\theta) \bigr)^2}
\Biggr| \dd t\\
&\leq
\gamma un\varphi_n \int_0^1
(\lambda_0+r) \,
\Biggl|
\frac{\lambda_0}{\lambda_0+\gamma ru \, \frac{\varphi_n}{\delta_n} \, x}
-
1
+
\ln \biggl(1+\frac{\gamma ru}{\lambda_0} \, \frac{\varphi_n}{\delta_n} \, x \biggr) \\*
&\phantom{{}=\gamma un\varphi_n \int_0^1 (\lambda_0+r) \Biggl|}
-
\frac{(\gamma ru)^2}{2} \, \frac{\varphi_n^2}{\delta_n^2} \, \frac{1}{\bigl( \lambda_0 +\gamma ru \, \frac{\varphi_n}{\delta_n} \, x\bigr)^2}
\Biggr|
\dd x \\
&=
\gamma un\varphi_n \int_0^1
(\lambda_0+r) \,
\Biggl|
\ln \biggl(1+\frac{\gamma ru}{\lambda_0} \, \frac{\varphi_n}{\delta_n} \, x \biggr)
-
\gamma ru \, \frac{\varphi_n}{\delta_n} \,
\frac{x}{\lambda_0+\gamma ru \, \frac{\varphi_n}{\delta_n} \, x} \\*
&\phantom{{}=\gamma un\varphi_n \int_0^1 (\lambda_0+r) \Biggl|}
-
\frac{(\gamma ru)^2}{2} \, \frac{\varphi_n^2}{\delta_n^2} \, \frac{1}{\bigl( \lambda_0+\gamma ru \, \frac{\varphi_n}{\delta_n} \, x\bigr)^2}
\Biggr|
\dd x.
\end{align*}
Since, for $0\leq x \leq 1$, we have
\begin{align*}
\Biggl|
\ln \biggl(1&+\frac{\gamma ru}{\lambda_0} \, \frac{\varphi_n}{\delta_n} \, x \biggr)
-\gamma ru \, \frac{\varphi_n}{\delta_n} \,
\frac{x}{\lambda_0+\gamma ru \, \frac{\varphi_n}{\delta_n} \, x}
-
\frac{(\gamma ru)^2}{2} \, \frac{\varphi_n^2}{\delta_n^2} \, \frac{1}{\bigl( \lambda_0+\gamma ru \, \frac{\varphi_n}{\delta_n} \, x\bigr)^2}
\Biggr| \\*
&\leq
\Biggl|
\ln \biggl(1+\frac{\gamma ru}{\lambda_0} \, \frac{\varphi_n}{\delta_n} \, x \biggr)
-
\frac{\gamma ru}{\lambda_0} \, \frac{\varphi_n}{\delta_n} \, x
\Biggr|
+
\Biggl|
\frac{\gamma ru}{\lambda_0} \, \frac{\varphi_n}{\delta_n} \, x
\biggl(
1
-
\frac{\lambda_0}{\lambda_0+\gamma ru \, \frac{\varphi_n}{\delta_n} \, x}
\biggr)
\Biggr| \\*
&\phantom{{}\leq}+
\frac{(\gamma ru)^2}{2} \, \frac{\varphi_n^2}{\delta_n^2} \, \frac{1}{\bigl( \lambda_0+\gamma ru \, \frac{\varphi_n}{\delta_n} \, x\bigr)^2} \\
&\leq
\frac12 \biggl( \frac{\gamma ru}{\lambda_0} \, \frac{\varphi_n}{\delta_n} \, x \biggr)^2
+
\frac{(\gamma ru)^2}{\lambda_0}\, \frac{\varphi_n^2}{\delta_n^2} \, \frac{x^2}{\lambda_0+\gamma ru \, \frac{\varphi_n}{\delta_n} \, x}
+
\frac{(\gamma ru)^2}{2} \, \frac{\varphi_n^2}{\delta_n^2} \,\frac{1}{\bigl( \lambda_0+\gamma ru \, \frac{\varphi_n}{\delta_n} \, x\bigr)^2}\\*
&\leq
2 \, \frac{(\gamma ru)^2}{\lambda_0^2}\, \frac{\varphi_n^2}{\delta_n^2} \, ,
\end{align*}
it comes
\[
D_1 \leq
2 \, \frac{(\lambda_0+r) (\gamma u)^3 r^2}{\lambda_0^2}\, \frac{n\varphi_n^3}{\delta_n^2}
\xrightarrow[n\to +\infty]{} 0.
\]

Proceeding similarly, it is not difficult to verify that $D_2$ and $D_3$ go to zero as well, and~so~\eqref{A32} is proved.
\end{proof}

Now, we can apply Theorem~2.1 of \citet{Kut98}, which yields that the family $\{ \Pb_\theta^{(n)}, \ \theta \in \Theta \}$ is LAN with rate
\[
\varphi'_n=Q_n^{-\frac{1}{2}} (\theta)=\gamma \varphi_n .
\]
More precisely, for all $u\in\UU'_n$, we have
\[
Z'_n(u) = \frac{\dd\Pb_{\theta+u\varphi'_n}\bigl(Y^{(n)}\bigr)}{\dd\Pb_{\theta}}
= \exp \Bigl\{ u\Delta'_n(\theta) - \frac{u^2}{2} + \varepsilon'_n(\theta, u) \Bigr\},
\]
where $\Delta'_n \Longrightarrow \xi_1$ is given by
\begin{align*}
\Delta'_n(\theta)
&=
Q_n^{-\frac{1}{2}}(\theta) \int_{0}^{\tau} q_n(\theta,t) \, \bigl[\mathrm{d} Y^{(n)}(t)-\Lambda^{(n)}_{\theta}(t) \dd t \bigr] \\*
&=
\gamma \sqrt{\frac{\delta_n}{n}} \int_{\theta}^{\theta+\delta_n} \frac{- \, \frac{r}{\delta_n}}{\lambda^{(n)}_{\theta}(t)} \, \bigl[\mathrm{d} Y^{(n)}(t)-n\lambda^{(n)}_{\theta}(t) \dd t \bigr] \\*
&=
- \frac{r\gamma}{\sqrt{n \delta_n}} \int_{\theta}^{\theta+\delta_n} \frac{1}{\lambda_0+\frac{r}{\delta_n} \, (t-\theta)} \, \dd Y^{(n)}(t)
+
r\gamma \sqrt{n\delta_n},
\end{align*}
and $\varepsilon'_n(\theta,u)$ converges to zero in probability.

Finally, it is clear that we can equivalently restate the LAN property using the rate
\[
\varphi_n=\frac{Q_n^{-\frac{1}{2}} (\theta)}{\gamma}=\sqrt{\frac{\delta_n}{n}}
\]
and the observations $X^{(n)}$. In this case, for all $u\in\UU_n$, we have
\[
Z_n(u) = \frac{\dd\Pb_{\theta+u\varphi_n}\bigl(X^{(n)}\bigr)}{\dd\Pb_{\theta}}
= \exp \Bigl\{ u\Delta_n(\theta) - \frac{u^2}{2}\, F + \varepsilon_n(\theta, u) \Bigr\},
\]
where $\Delta_n \Longrightarrow \xi_F$ is given by
\[
\Delta_n(\theta) =\frac{\Delta'_n(\theta)}{\gamma} =
- \frac{r}{\sqrt{n \delta_n}} \; \sum_{j=1}^{n} \int_{\theta}^{\theta+\delta_n} \frac{1}{\lambda_0+\frac{r}{\delta_n} \, (t-\theta)} \, \dd X_j(t)
+
r\sqrt{n\delta_n},
\]
and $\varepsilon_n(\theta,u) = \varepsilon'_n \bigl( \theta, \frac{u}{\gamma} \bigr)$ converges to zero in probability.

To wrap up this part, let us note that it is not very difficult to adapt the proofs of the conditions $(A_2)$ and $(A_3)$ to the case where $\psi$ is any strictly positive continuous (not necessarily constant) function on~$[0,\tau]$. In this case, taking
\[
q_n(\theta,t)
=
\frac{-\frac{r}{\delta_n}}{\psi(\theta) + \frac{r}{\delta_n} \, (t-\theta)} \,\ind_{\bil[\theta, \theta+\delta_n\bir]} (t),
\]
we get
\begin{align*}
Q_n(\theta)
&=
\int_0^\tau q_n^2(\theta,t) \Lambda_{\theta}^{(n)} (t) \dd t \\*
&=
\frac{n r^2}{\delta_n^2} \int_\theta^{\theta+\delta_n} \frac{\psi(t) + \frac{r}{\delta_n} \, (t-\theta)}{\Bigl(\psi(\theta) + \frac{r}{\delta_n} \, (t-\theta) \Bigr)^2} \, \dd t \\
&=
\frac{n r}{\delta_n} \int_0^r \frac{\psi \bigl( \theta +\frac{\delta_n}{r} x \bigr) +x}{(\psi(\theta) + x)^2} \, \dd x\\
&=
F\varphi_n^{-2} + r \varphi_n^{-2} \int_0^r \frac{\psi \bigl( \theta +\frac{\delta_n}{r} x \bigr) - \psi(\theta)}{(\psi(\theta) + x)^2} \, \dd x \\*
&=
F\varphi_n^{-2} \bigl(1+o(1)\bigr) .
\end{align*}

Note also, that since the function $\psi$ is continuous and strictly positive on $[0,\tau]$, it admits a minimum $m>0$ and a maximum $M>0$, which can replace $\lambda_0$ in different estimates used in the proofs of the condition $(A_3)$. For the case $r<0$, it is equally important to remind that we supposed that $r>-m$, and hence $M+r>m+r>0$.

Then, all the proofs can be easily adapted, and so we obtain Lemma~\ref{LAN} and, consequently,~Theorem~\ref{Borne.CL}.

\paragraph{}
Now, let us turn to the proof of Theorem~\ref{MLE-BE.CL}. As it was already noted above, it can be proved by applying Theorems~1.10.1 and~1.10.2 of \citet{IbHas81}. So, it is sufficient to check the conditions of these theorems. Since the convergence of finite-dimensional distributions of the normalized likelihood ratio process $Z_n$ to those of the process~$Z^\circ_F$ follows from the already established LAN property, it remains to prove the following two lemmas.
\begin{lemma}
Suppose\/ $n\delta_n\to+\infty$. Then, there exists a constant\/ $C>0$, such that for\/ $n$ sufficiently large, we have
\begin{equation}
\Ex_\theta \bigl\vert Z_n^{1/2} (u)- Z_n^{1/2} (v) \bigr\vert ^2 \leq C\bil| u - v \bir|^2
\label{C2.CL-eq}
\end{equation}
for all\/ $u,v \in \UU_n$ and\/ $\theta \in \Theta$.
\label{C2.CL}
\end{lemma}
\begin{lemma}
Suppose\/ $n\delta_n\to+\infty$. Then, there exists a constant\/ $\kappa>0$, such that for\/ $n$ sufficiently large, we have
\begin{equation}
\Ex_\theta Z_n^{1/2}(u) \leq \exp\bigl\{-\kappa \min \{\bil| u \bir|, u^2\} \bigr\}
\label{C3.CL-eq}
\end{equation}
for all\/ $u\in \UU_n$ and\/ $\theta \in \Theta$.
\label{C3.CL}
\end{lemma}
\begin{proof}[Proof of Lemma~\ref{C2.CL}]
First of all, let us note that, as for $\bil| u - v \bir| \geq 1$ we have
\[
\Ex_\theta \bigl\vert Z_n^{1/2} (u)- Z_n^{1/2} (v) \bigr\vert ^2 \leq 4\leq 4\bil| u - v \bir|^2,
\]
it is sufficient to consider the case $\bil| u - v \bir| \leq 1$. Moreover, without loss of generality, we can suppose that $u<v$.

Further we consider the case $r>0$ only (the case $r<0$ can be treated in a similar way). Using Lemma~1.5 of \citet{Kut98}, we have
\begin{align*}
\Ex_\theta \bigl\vert Z_n^{1/2} (u)-Z_n^{1/2} (v) \bigr\vert ^2
&\leq
n\int_0^\tau \Bigl( \sqrt{\lambda_{\theta+u\varphi_n}(t)} - \sqrt{\lambda_{\theta+v\varphi_n}(t)} \mskip 2mu \Bigr)^2 \dd t\\*
&=
n\int_0^\tau \frac{\bigl( \lambda_{\theta+u\varphi_n}(t) - \lambda_{\theta+v\varphi_n}(t) \bigr)^2}{\bigl( \sqrt{\lambda_{\theta+u\varphi_n}(t)} + \sqrt{\lambda_{\theta+v\varphi_n}(t)} \mskip 2mu \bigr)^2} \, \dd t \\
&\leq
\frac{n}{4\lambda_0} \int_0^\tau \bigl( \lambda_{\theta+u\varphi_n}(t) - \lambda_{\theta+v\varphi_n}(t) \bigr)^2 \dd t \\*
&=
\frac{n}{4\lambda_0} \int_0^\tau f_n(t) \dd t,
\end{align*}
where we denoted $f_n(t)=\bigl( \lambda_{\theta+u\varphi_n}(t) - \lambda_{\theta+v\varphi_n}(t) \bigr)^2$.

Taking into account that $\delta_n/\varphi_n=\sqrt{n\delta _n}\to+\infty$, for $n$ sufficiently large (namely, such that~$n\delta_n\geq 1$) we have $(v-u)\varphi_n\leq \varphi_n \leq \delta_n$, and hence $v\varphi_n\leq u\varphi_n+\delta_n$. Therefore, we can write
\begin{align*}
\int_0^\tau f_n(t) \dd t &=
\int_{\theta+u\varphi_n}^{\theta+v\varphi_n} f_n(t) \dd t + \int_{\theta+v\varphi_n }^{\theta+u\varphi_n+\delta_n} f_n(t) \dd t + \int_{\theta+u\varphi_n +\delta_n}^{\theta+v\varphi_n+\delta_n} f_n(t) \dd t \\*
&=
E_1+E_2+E_3
\end{align*}
with evident notations.

For $E_1$, we have
\[
E_1 =
\int_{\theta+u\varphi_n}^{\theta+v\varphi_n} \Bigl( \frac{r}{\delta_n} ( t-\theta - u\varphi_n ) \Bigr)^2 \dd t =
\frac{r^2}{\delta_n^2} \, \frac{( v\varphi_n - u\varphi_n )^3}{3} =
\frac{r^2}{3} \, \frac{\varphi_n^3}{\delta_n^2} \, (v-u)^3,
\]
and proceeding similarly, we get
\[
E_2 =
r^2 \, \frac{\varphi_n^2}{\delta_n} \, (v-u)^2 -
r^2 \, \frac{\varphi_n^3}{\delta_n^2} \, (v-u)^3
\qquad
\text{and}
\qquad
E_3 =
\frac{r^2}{3} \, \frac{\varphi_n^3}{\delta_n^2} \, (v-u)^3.
\]
Thus, using the fact that $n\varphi_n^2/\delta_n=1$, we have
\[
\Ex_\theta \bigl\vert Z_n^{1/2} (u)- Z_n^{1/2} (v) \bigr\vert ^2 \leq
\frac{r^2}{4\lambda_0} \, \frac{n\varphi_n^2}{\delta_n} \, (v-u)^2 -
\frac{r^2}{12\lambda_0} \, \frac{n\varphi_n^3}{\delta_n^2} \, (v-u)^3
\leq \frac{r^2}{4\lambda_0} \, (v-u)^2,
\]
and so the inequality~\eqref{C2.CL-eq} is proved with $C = \max\bigl\{4,\frac{r^2}{4\lambda_0}\bigr\}$.
\end{proof}
\begin{proof}[Proof of Lemma~\ref{C3.CL}]
We consider the case where $u>0$ and $r>0$ only (the other cases can be treated in a similar way). Using Lemma~1.5 of \citet{Kut98}, we have
\begin{align*}
\Ex_\theta Z_n^{1/2}(u)
&=
\exp\biggl\{-\frac{n}{2} \int_0^\tau \Bigl( \sqrt{\lambda_{\theta+u\varphi_n}(t)} - \sqrt{\lambda_{\theta}(t)} \mskip 2mu \Bigr)^2 \dd t \biggr\}\\*
&=
\exp\biggl\{-\frac{n}{2} \int_0^\tau \frac{\bigl( \lambda_{\theta+u\varphi_n}(t) - \lambda_{\theta}(t) \bigr)^2}{\bigl( \sqrt{\lambda_{\theta+u\varphi_n}(t)} + \sqrt{\lambda_{\theta}(t)} \mskip 2mu \bigr)^2} \, \dd t\biggr\} \\
&\leq
\exp\biggl\{-\frac{n}{8(\lambda_0+r)} \int_0^\tau \bigl( \lambda_{\theta+u\varphi_n}(t) - \lambda_{\theta}(t) \bigr)^2 \dd t\biggr\} \\*
&=
\exp\biggl\{-\frac{n}{8(\lambda_0+r)} \int_0^\tau g_n(t) \dd t\biggr\},
\end{align*}
where we denoted $g_n(t)=\bigl( \lambda_{\theta+u\varphi_n}(t) - \lambda_{\theta}(t) \bigr)^2$.

Now we treat separately two cases: $u\leq\delta_n/\varphi_n$ and $u\geq\delta_n/\varphi_n$. In the first case, the situation is similar to that of Figure~\ref{fig-int-slow-case}, and so we obtain
\begin{align*}
\int_0^\tau g_n(t) \dd t
&=
\int_\theta^{\theta+u\varphi_n} g_n(t) \dd t +
\int_{\theta+u\varphi_n}^{\theta+\delta_n} g_n(t) \dd t +
\int_{\theta+\delta_n}^{\theta+u\varphi_n +\delta_n} g_n(t) \dd t \\*
&=
J_1+J_2+J_2
\end{align*}
with evident notations.

For $J_1$, we have
\[
J_1 =
\int_\theta^{\theta+u\varphi_n} \Bigl( \frac{r}{\delta_n} \, (t-\theta) \Bigr)^2 \dd t =
\frac{r^2}{3} \, \frac{\varphi_n^3}{\delta_n^2} \, u^3,
\]
and proceeding similarly, we obtain
\[
J_2 =
r^2 \, \frac{\varphi_n^2}{\delta_n} \, u^2 -
r^2 \, \frac{\varphi_n^3}{\delta_n^2} \, u^3
\qquad
\text{and}
\qquad
J_3 =
\frac{r^2}{3} \, \frac{\varphi_n^3}{\delta_n^2} \, u^3.
\]
Thus, using the fact that $u\leq\delta_n/\varphi_n$, we get
\begin{align}
\Ex_\theta Z_n^{1/2}(u)
&\leq
\exp\biggl\{-\frac{r^2}{8(\lambda_0+r)} \, \frac{n\varphi_n^2}{\delta_n} \, u^2 + \frac{r^2}{24(\lambda_0+r)} \, \frac{n\varphi_n^3}{\delta_n^2} \, u^3 \biggr\} \notag \\*
&\leq
\exp\biggl\{-\frac{r^2}{12(\lambda_0+r)} \, \frac{n\varphi_n^2}{\delta_n} \, u^2 \biggr\} \notag \\*
&\leq
\exp\biggl\{-\frac{r^2}{12(\lambda_0+r)} \, \frac{n\varphi_n^2}{\delta_n} \, \min \{u, u^2\} \biggr\},
\label{L3-eq1}
\end{align}
and recalling that $n\varphi_n^2/\delta_n=1$, we conclude that
\[
\Ex_\theta Z_n^{1/2}(u) \leq \exp\biggl\{-\frac{r^2}{12(\lambda_0+r)} \, \min\{u, u^2\} \biggr\}.
\]

In the second case ($u\geq \delta_n/\varphi_n$), the situation is that of Figure~\ref{fig-int-fast-case}, and so we obtain
\begin{align*}
\int_0^\tau g_n(t) \dd t
&=
\int_\theta^{\theta+\delta_n} g_n(t) \dd t +
\int_{\theta+\delta_n}^{\theta+ u\varphi_n} g_n(t) \dd t +
\int_{\theta+ u\varphi_n}^{\theta+u\varphi_n +\delta_n} g_n(t) \dd t \\*
&=
J'_1+J'_2+J'_2
\end{align*}
with evident notations.

For $J_1'$, we have
\[
J_1' =
\int_\theta^{\theta+\delta_n} \Bigl( \frac{r}{\delta_n} \, (t-\theta) \Bigr)^2 \dd t =
\frac{r^2}{3} \, \delta_n,
\]
and proceeding similarly, we obtain
\[
J_2' =
r^2 u \varphi_n -
r^2 \delta_n
\qquad
\text{and}
\qquad
J_3' =
\frac{r^2}{3} \, \delta_n.
\]
Thus, using the fact that $\delta_n\leq u\varphi_n$, we get
\begin{align}
\Ex_\theta Z_n^{1/2}(u)
&\leq
\exp\biggl\{-\frac{r^2}{8(\lambda_0+r)} \, n\varphi_n u + \frac{r^2}{24(\lambda_0+r)} \, n\delta_n \biggr\} \notag \\*
&\leq
\exp\biggl\{-\frac{r^2}{12(\lambda_0+r)} \, n\varphi_n u \biggr\} \notag \\*
&\leq
\exp\biggl\{-\frac{r^2}{12(\lambda_0+r)} \, n\varphi_n \min\{u, u^2\} \biggr\},
\label{L3-eq2}
\end{align}
and taking into account that $n\varphi_n=\sqrt{n\delta_n}\to +\infty$, for $n$ sufficiently large (namely, such that~$n\delta_n\geq 1$) we conclude again that
\[
\Ex_\theta Z_n^{1/2}(u) \leq \exp\biggl\{-\frac{r^2}{12(\lambda_0+r)} \, \min\{u, u^2\} \biggr\}.
\]
So the inequality~\eqref{C3.CL-eq} is proved with $\kappa =\frac{r^2}{12(\lambda_0+r)}\,$.
\end{proof}

Note that we can easily adapt the proofs of Lemmas~\ref{C2.CL} and~\ref{C3.CL} to the case where $\psi$ is any strictly positive continuous (not necessarily constant) function on $[0,\tau]$, and so Theorem~\ref{MLE-BE.CL} is proved.

\subsubsection*{Fast case}

As it was already noted above, Theorems~\ref{Borne.CR} and~\ref{BE.CR} can be proved by applying Theorems~1.9.1 and~1.10.2 of \citet{IbHas81}. So, it is sufficient to check the conditions of these theorems, that is, to prove the following three lemmas.

\begin{lemma}
Suppose\/ $n\delta_n\to 0$. Then, the finite-dimensional distributions of the process\/~$Z_n$ converge to those of the process\/ $Z^\star_{a,b}$ with\/ $a=\psi(\theta)$ and\/ $b=\psi(\theta)+r$.
\label{fidi}
\end{lemma}
\begin{lemma}
Suppose\/ $n\delta_n\to 0$. Then, there exists a constant\/ $C>0$, such that
\begin{equation}
\Ex_\theta \bigl\vert Z_n^{1/2} (u)- Z_n^{1/2} (v) \bigr\vert ^2 \leq C\bil| u - v \bir|
\label{C2.CR-eq}
\end{equation}
for all\/ $n\in\NN$, $u,v \in \UU_n$ and\/ $\theta \in \Theta$.
\label{C2.CR}
\end{lemma}
\begin{lemma}
Suppose\/ $n\delta_n\to 0$. Then, there exists a constant\/ $\kappa>0$, such that for\/ $n$ sufficiently large, we have
\begin{equation}
\Ex_\theta Z_n^{1/2}(u) \leq \exp\bigl\{-\kappa \min \{\bil| u \bir|, u^2\} \bigr\}
\label{C3.CR-eq}
\end{equation}
for all\/ $u\in \UU_n$ and\/ $\theta \in \Theta$.
\label{C3.CR}
\end{lemma}

Before proving Lemmas~\ref{fidi}--\ref{C3.CR}, note that since we are in the fast case, for any $u\in\RR$, we have~$u\varphi_n > \delta_n$ for  $n$ sufficiently large. We present in Figure~\ref{fig-int-fast-case} the functions $\lambda_\theta$ and $\lambda_{\theta_u}$, where $\theta_u = \theta + u\varphi_n$, with $u,r>0$ and $n \gg 1$.
\begin{figure}[!ht]
\centering
\includegraphics[scale=0.36]{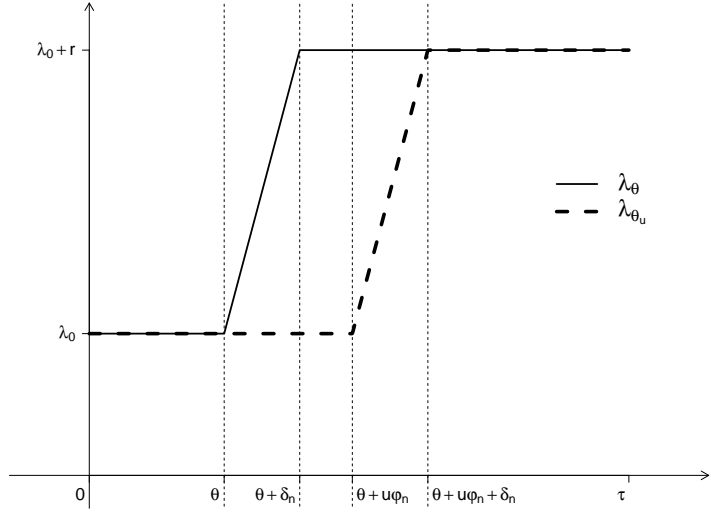}
\caption{$\lambda_\theta$ and $\lambda_{\theta_u}$ with $u,r>0$ and $n \gg 1$ (fast case)}
\label{fig-int-fast-case}
\end{figure}
\begin{proof}[Proof of Lemma~\ref{fidi}]
We study the convergence of 2-dimensional distributions only (the convergence of $d$-dimensional distributions for $d\geq 3$ can be treated in a similar way). For this, let us fix some $u,v \in \RR$ and consider the distribution of the vector $\bigl( \ln Z_n(u),\ln Z_n(v)\bigr)$, where $n$ is sufficiently large, so that $u,v \in \UU_n$. Its characteristic function is given, for all~$x, y \in \RR$, by
\begin{align*}
&\phi_{\bigl( \ln Z_n(u),\ln Z_n(v)\bigr)} \bigl( x,y \bigr)=
\Ex_\theta \exp \bigl\{ i x\ln Z_n(u)+ i y\ln Z_n(v) \bigr\} \\*
&\qquad=
\Ex_\theta \exp \Biggl\{
ix \sum_{j=1}^{n} \biggl[ \int_{0}^{\tau} \ln\biggl(\frac{\lambda_{\theta+u\varphi_n}(t)}{\lambda_{\theta}(t)} \biggr)\dd X_j(t) - \int_{0}^{\tau} \bigl(\lambda_{\theta+u\varphi_n}(t) - \lambda_{\theta}(t) \bigr) \dd t \bigr) \biggr] \\*
&\qquad \phantom{{}=\Ex_\theta \exp \Biggl\{}+
iy \sum_{j=1}^{n} \biggl[ \int_{0}^{\tau} \ln\biggl(\frac{\lambda_{\theta+v\varphi_n}(t)}{\lambda_{\theta}(t)} \biggr)\dd X_j(t) - \int_{0}^{\tau} \bigl(\lambda_{\theta+v\varphi_n}(t) - \lambda_{\theta}(t) \bigr) \dd t \bigr) \biggr]
\Biggr\}\\
&\qquad=
\Ex_\theta \exp \Biggl\{
ix \sum_{j=1}^{n} \biggl[ \int_{0}^{\tau} \ln\biggl(\frac{\lambda_{\theta+u\varphi_n}(t)}{\lambda_{\theta}(t)} \biggr)\dd X_j(t) + ru\varphi_n \biggr] \\*
&\qquad \phantom{{}=\Ex_\theta \exp \Biggl\{} +
iy \sum_{j=1}^{n} \biggl[ \int_{0}^{\tau} \ln\biggl(\frac{\lambda_{\theta+v\varphi_n}(t)}{\lambda_{\theta}(t)} \biggr)\dd X_j(t) +rv\varphi_n \biggr]
\Biggr\} \\
&\qquad=
\exp\bigl\{nir(ux+vy)\varphi_n\bigr\} \\*
&\qquad \phantom{{}=} \times
\Ex_\theta \exp \Biggl\{ ix \sum_{j=1}^{n} \int_{0}^{\tau} \ln\biggl(\frac{\lambda_{\theta+u\varphi_n}(t)}{\lambda_{\theta}(t)} \biggr)\dd X_j(t)
+
iy \sum_{j=1}^{n} \ \int_{0}^{\tau} \ln\biggl(\frac{\lambda_{\theta+v\varphi_n}(t)}{\lambda_{\theta}(t)} \biggr)\dd X_j(t)
\Biggr\} \\
&\qquad=
\exp\bigl\{ir(ux+vy)\bigr\} \\*
&\qquad \phantom{{}=} \times \exp \Biggl\{n\int_{0}^{\tau} \Biggl(
\exp\biggl\{ ix\ln \biggl( \frac{\lambda_{\theta+u\varphi_n}(t)}{\lambda_{\theta}(t)} \biggr)
+
iy \ln \biggl(\frac{\lambda_{\theta+v\varphi_n}(t)}{\lambda_{\theta}(t)}\biggr) \biggr\}
-1 \Biggr) \lambda_{\theta}(t) \dd t \Biggr \} \\*
&\qquad=
\exp\bigl\{ir(ux+vy)\bigr\} \exp\biggl\{n\int_{0}^{\tau} f_{n}(t) \dd t\biggr\}
\end{align*}
with an evident notation.

We consider the case where $v>u\geq0$ and $r>0$ only (the other cases can be treated in a similar way). In this case, for $n$ sufficiently large, we have $\delta_n<u\varphi_n$ and $u\varphi_n+\delta_n <v\varphi_n$, and so we can write
\begin{align*}
\int_0^\tau f_{n}(t) \dd t
&=
\int_{\theta}^{\theta +\delta_n} f_{n}(t) \dd t
+
\int_{\theta+\delta_n}^{\theta+u\varphi_n} f_{n}(t) \dd t
+
\int_{\theta+u\varphi_n}^{\theta+u\varphi_n+\delta_n} f_{n}(t) \dd t \\*
&\phantom{{}=} +
\int_{\theta+u\varphi_n+\delta_n}^{\theta+v\varphi_n} f_{n}(t) \dd t +
\int_{\theta+v\varphi_n}^{\theta+v\varphi_n+\delta_n} f_{n}(t) \dd t \\*
&=
\sum_{j=1}^5 I_j
\end{align*}
with evident notations.

For $I_1$, using the change of variable
\[
s=\frac{t-\theta}{\delta_n} \, ,
\]
we have
\begin{align*}
I_1
&=
\int_{\theta}^{\theta +\delta_n} f_{n}(t) \dd t \\*
&=
\int_{\theta}^{\theta +\delta_n}
\Biggl(
\exp \biggl\{ ix\ln \biggl( \frac{\lambda_{\theta+u\varphi_n}(t)}{\lambda_{\theta}(t)} \biggr)
+
iy \ln \biggl(\frac{\lambda_{\theta+v\varphi_n}(t)}{\lambda_{\theta}(t)}\biggr) \biggr\}
-1 \Biggr) \lambda_{\theta}(t) \dd t \\
&=
\int_{\theta}^{\theta+\delta_n}
\Biggl( \exp \biggl\{ i(x+y) \ln \biggl(\frac{\lambda_0}{\lambda_0+\frac{r}{\delta_n} \, (t-\theta)}\biggr) \biggr\}-1 \Biggr) \biggl( \lambda_0+\frac{r}{\delta_n} \, (t-\theta) \biggr) \dd t\\*
&=
\delta_n \int_0^1 \biggl( \exp \Bigl\{i(x+y) \ln \Bigl( \frac{\lambda_0}{\lambda_0+rs} \Bigr)\Bigr\} -1 \biggr) (\lambda_0+rs) \dd s = c_1 \delta_n ,
\end{align*}
where $c_1$ is some constant, and proceeding similarly, we obtain $I_3=c_2 \delta_n$ and $I_5=c_3 \delta_n$.

For $I_2$, we have
\begin{align*}
I_2
&=
\int_{\theta+\delta_n}^{\theta+u\varphi_n} f_{n}(t) \dd t \\*
&=
\int_{\theta+\delta_n}^{\theta+u\varphi_n}
\biggl( \exp \Bigl\{ i(x+y)\ln \Bigl(\frac{\lambda_0}{\lambda_0+r}\Bigr) \Bigr\}-1 \biggr) (\lambda_0+r) \dd t \\*
&=
\biggl( \exp \Bigl\{ i(x+y)\ln \Bigl(\frac{\lambda_0}{\lambda_0+r}\Bigr) \Bigr\}-1 \biggr) (\lambda_0+r) (u\varphi_n - \delta_n ),
\end{align*}
and proceeding similarly, we obtain
\[
I_4=\biggl( \exp \Bigl\{ i y \ln \Bigl(\frac{\lambda_0}{\lambda_0+r}\Bigr) \Bigr\}-1 \biggr) (\lambda_0+r) \bigl((v-u)\varphi_n - \delta_n\bigr).
\].

Therefore, it comes
\begin{multline*}
\int_0^\tau f_{n}(t) \dd t=
C\delta_n
+
\biggl[ (v-u) \biggl( \exp \Bigl\{ iy \ln \Bigl( \frac{\lambda_0}{\lambda_0+r} \Bigr) \Bigr\} -1 \biggr) \\*
+u\biggl( \exp \Bigl\{ i(x+y)\ln \Bigl(\frac{\lambda_0}{\lambda_0+r}\Bigr) \Bigr\}-1 \biggr) \biggr]
(\lambda_0+r) \varphi_n,
\end{multline*}
and hence, recalling that $\varphi_n=1/n$, we get
\begin{align*}
\exp \biggl\{n\int_0^\tau f_{n}(t) \dd t \biggr\}
&=
\exp \{ C n\delta_n \}
\exp \biggl\{
u(\lambda_0+r)\biggl( \exp \Bigl\{ i(x+y)\ln \Bigl(\frac{\lambda_0}{\lambda_0+r}\Bigr) \Bigr\}-1 \biggr)
\biggr\} \\*
&\phantom{{}=}\times \exp \biggl\{ (v-u)(\lambda_0+r) \biggl( \exp \Bigl\{ iy \ln \Bigl( \frac{\lambda_0}{\lambda_0+r} \Bigr) \Bigr\} -1 \biggr)
\biggr\}.
\end{align*}
Since
\[
n\delta_n \xrightarrow[n\to +\infty]{} 0,
\]
we finally conclude that
\begin{align}
\phi_{\bigl( \ln Z_n(u),\ln Z_n(v)\bigr)}(x,y)
&\xrightarrow[n\to +\infty]{}
\exp \{ir(ux+vy) \} \notag \\*
&\phantom{{}\xrightarrow[n\to +\infty]{}}\times
\exp \biggl\{
u(\lambda_0+r) \biggl( \exp \Bigl\{ i(x+y)\ln \Bigl(\frac{\lambda_0}{\lambda_0+r}\Bigr) \Bigr\}-1 \biggr)
\biggr\} \notag \\*
&\phantom{{}\xrightarrow[n\to +\infty]{}}\times\exp \biggl\{ (v-u)(\lambda_0+r) \biggl( \exp \Bigl\{ iy \ln \Bigl( \frac{\lambda_0}{\lambda_0+r} \Bigr) \Bigr\} -1 \biggr)
\biggr\}.
\label{ChF_fi-di}
\end{align}

Now, let us calculate, still for the case $v>u\geq 0$, the characteristic function of the vector $\bigl( \ln Z^\star_{a,b}(u),\ln Z^\star_{a,b}(v)\bigr)$. For all $x$, $y\in \RR$, we have
\begin{align*}
\phi_{\bigl( \ln Z^\star_{a,b}(u),\ln Z^\star_{a,b}(v)\bigr)}(x,y)
&=
\Ex \exp \bigl\{ ix\ln Z^\star_{a,b}(u) + iy\ln Z^\star_{a,b}(v) \bigr\} \\*
&=
\Ex \exp \bigg\{ ix\biggl( \ln \Bigl( \frac{a}{b} \Bigr)Y^+(u) + (b-a)u \biggr) \\*
&\phantom{{}=\Ex \exp \bigg\{ }
+iy \biggl( \ln \Bigl( \frac{a}{b} \Bigr)Y^+(v) +(b-a)v \biggr) \biggr\} \\
&=
\exp\{i(b-a) (ux+vy) \} \\*
&\phantom{{}=} \times
\Ex \exp \Bigl\{ ix\ln \Bigl( \frac{a}{b} \Bigr)Y^+(u) +iy\ln \Bigl( \frac{a}{b} \Bigr)Y^+(v)\Bigr\}.
\end{align*}
Since $Y^+$ is a Poisson process and $v>u\geq 0$, we obtain
\begin{align*}
\Ex \exp & \Big\{ ix\ln \Bigl( \frac{a}{b} \Bigr)Y^+(u) +iy\ln \Bigl( \frac{a}{b} \Bigr)Y^+(v)\Bigr\} \\*
&=
\Ex \exp \Big\{ i(x+y)\ln \Bigl( \frac{a}{b} \Bigr)Y^+(u) \Bigr\} \, \,
\Ex \exp \Big\{ iy\ln \Bigl( \frac{a}{b} \Bigr)\bigl( Y^+(v) - Y^+(u) \bigr)\Bigr\} \\*
&=
\exp \biggl\{
ub\biggl( \exp \Bigl\{ i(x+y)\ln \Bigl(\frac{a}{b}\Bigr) \Bigr\}-1 \biggr)
\biggr\}\exp \biggl\{ (v-u)b \biggl( \exp \Bigl\{ iy \ln \Bigl( \frac{a}{b} \Bigr) \Bigr\} -1 \biggr)
\biggr\}.
\end{align*}
Therefore, we get
\begin{align*}
\phi_{\bigl( \ln Z^\star_{a,b}(u),\ln Z^\star_{a,b}(v)\bigr)}(x,y)
&=
\exp\{i(b-a) (ux+vy) \}
\exp \biggl\{
ub\biggl( \exp \Bigl\{ i(x+y)\ln \Bigl(\frac{a}{b}\Bigr) \Bigr\}-1 \biggr)
\biggr\} \\*
&\phantom{{}=}\times
\exp \biggl\{ (v-u)b \biggl( \exp \Bigl\{ iy \ln \Bigl( \frac{a}{b} \Bigr) \Bigr\} -1 \biggr)
\biggr\},
\end{align*}
which, taking $a=\psi(\theta)=\lambda_0$ and $b=\psi(\theta)+r=\lambda_0+r$, is the same as the right-hand side of~\eqref{ChF_fi-di}. This shows that $\bigl( \ln Z_n(u),\ln Z_n(v)\bigr)$ converge to $\bigl( \ln Z^\star_{\lambda_0,\lambda_0+r}(u),\ln Z^\star_{\lambda_0,\lambda_0+r}(v)\bigr)$, and hence the convergence of 2-dimensional distributions of $Z_n$ to those of $Z^\star_{\lambda_0,\lambda_0+r}$ is proved.
\end{proof}
\begin{proof}[Proof of Lemma~\ref{C2.CR}]
We consider the case $r>0$ only (the case $r<0$ can be treated in a similar way) and, without loss of generality, we can suppose that $u<v$. According to Lemma~1.5 of \citet{Kut98}, we have
\[
\Ex_\theta \bigl\vert Z_n^{1/2} (u)-Z_n^{1/2} (v) \bigr\vert ^2
\leq
n\int_0^\tau \Bigl( \sqrt{\lambda_{\theta+u\varphi_n}(t)} - \sqrt{\lambda_{\theta+v\varphi_n}(t)} \mskip 2mu \Bigr)^2 \dd t.
\]
Since
\[
\lambda_{\theta+u\varphi_n}(t) \geq \lambda_{\theta+v\varphi_n}(t)
\]
and
\[
\sqrt{\lambda_{\theta+u\varphi_n}(t)} - \sqrt{\lambda_{\theta+v\varphi_n}(t)}
\leq
\sqrt{\lambda_{\theta+u\varphi_n}(t)-\lambda_{\theta+v\varphi_n}(t)} \, ,
\]
we obtain
\[
\int_0^\tau \Bigl( \sqrt{\lambda_{\theta+u\varphi_n}(t)} - \sqrt{\lambda_{\theta+v\varphi_n}(t)} \mskip 2mu \Bigr)^2 \dd t
\leq
\int_0^\tau \bigl( \lambda_{\theta+u\varphi_n}(t)-\lambda_{\theta+v\varphi_n}(t)\bigr) \dd t.
\]
By a simple area calculation, we get
\[
\int_0^\tau \bigl( \lambda_{\theta+u\varphi_n}(t)-\lambda_{\theta+v\varphi_n}(t) \bigr) \dd t
=
r(v-u)\varphi_n,
\]
which, taking into account that $\varphi_n=1/n$, yields the inequality~\eqref{C2.CR-eq} with $C=r$.
\end{proof}
\begin{proof}[Proof of Lemma~\ref{C3.CR}]
We consider the case where $u>0$ and $r>0$ only (the other cases can be treated in a similar way).

As in the proof of Lemma~\ref{C3.CL}, we treat separately two cases: $u\leq\delta_n/\varphi_n$ and $u\geq\delta_n/\varphi_n$. In the first case, we have already shown (see~\eqref{L3-eq1}, the proof of which is valid also in the fast case) that
\[
\Ex_\theta Z_n^{1/2}(u) \leq \exp\biggl\{-\frac{r^2}{12(\lambda_0+r)} \, \frac{n\varphi_n^2}{\delta_n} \, \min\{u, u^2\} \biggr\}.
\]
Thus, taking into account that $\frac{n\varphi_n^2}{\delta_n} = \frac{1}{n\delta_n} \to +\infty$, for $n$ sufficiently large (namely, such that~$n\delta_n\leq 1$) we conclude that
\[
\Ex_\theta Z_n^{1/2}(u) \leq \exp\biggl\{-\frac{r^2}{12(\lambda_0+r)} \, \min\{u, u^2\} \biggr\}.
\]

In the second case ($u\varphi_n\geq \delta_n$), we have already shown (see~\eqref{L3-eq2}, the proof of which is valid also in the fast case) that
\[
\Ex_\theta Z_n^{1/2}(u) \leq \exp\biggl\{-\frac{r^2}{12(\lambda_0+r)} \, n\varphi_n \min\{u, u^2\} \biggr\}.
\]
Thus, recalling that $\varphi_n=1/n$, we conclude again that
\[
\Ex_\theta Z_n^{1/2}(u) \leq \exp\biggl\{-\frac{r^2}{12(\lambda_0+r)} \, \min\{u, u^2\} \biggr\}.
\]
So the inequality~\eqref{C3.CR-eq} is proved with $\kappa =\frac{r^2}{12(\lambda_0+r)}\,$.
\end{proof}

Note that we can easily adapt the proofs of Lemmas~\ref{fidi}--\ref{C3.CR} to the case where $\psi$ is any strictly positive continuous (not necessarily constant) function on $[0,\tau]$, and so Theorems~\ref{Borne.CR} and~\ref{BE.CR} are proved.

\section*{Acknowledgments}

This research was financially supported by RFBR and CNRS (project 20--51--15001). The authors equally acknowledge support from the Labex CEMPI (ANR--11--LABX--0007--01). Finally, the authors would like to thank the editors and the reviewers for their constructive comments, which helped to improve the manuscript.

\bibliographystyle{myspbasic}
\bibliography{bibSCP}

\end{document}